\documentclass[a4paper,11pt]{article}

\usepackage{amsmath,mathtools,amsthm,amssymb,amsfonts}
\usepackage[numbers,sort&compress]{natbib}
\usepackage[colorlinks,linkcolor=blue,anchorcolor=blue,citecolor=green]{hyperref}
\usepackage{enumerate}
\numberwithin{equation}{section}

\newcommand{\udiv}{\mathrm{div}}
\allowdisplaybreaks[4]
\newcommand{\udt}{\frac{\mathrm{d}}{\mathrm{d}t}}
\newtheorem{theorem}{Theorem}[section]

\newtheorem{lemma}{Lemma}[section]
\newtheorem{rem}{Remark}[section]

\newtheorem{corollary}{Corollary}[section]
\newtheorem{pro}{Proposition}[section]
\makeatletter%使用罗马数字

\newcommand{\Rmnum}[1]{\expandafter\@slowromancap\romannumeral #1@}
\usepackage{authblk}
\usepackage[title]{appendix}
\makeatother

\begin{document}
	\title{The Cauchy problem for an inviscid and non-diffusive Oldroyd-{B} model in two dimensions}
	
	\author[a]{Yuanzhi Tu}
	\author[a]{Yinghui Wang}
	\author[a]{Huanyao Wen \thanks{Corresponding author.}}
	\affil[a]{School of Mathematics, South China University of Technology, Guangzhou, China}
	
	%\author{{Y{\sc uanzhi}  T{\sc u}\footnote{School of Mathematics, South China University of Technology, Guangzhou 510641, P.R. China. e-mail:	yztumath@163.com.}, \quad Y{\sc inghui}  W{\sc ang}\footnote{School of Mathematics, South China	University of Technology, Guangzhou 510641, P.R. China. e-mail: yhwangmath@163.com.}, \quad H{\sc uanyao}  W{\sc en}\footnote{School of Mathematics, South China	University of Technology, Guangzhou 510641, P.R. China. e-mail: mahywen@scut.edu.cn.}}}

	\date{}
	\maketitle
	\renewcommand{\thefootnote}{}
	
	\footnote{ {E}-mail: mayztu@mail.scut.edu.cn(Tu); yhwangmath@163.com(Wang); mahywen@scut.edu.cn(Wen).}
	
	\begin{abstract}
		A two-dimensional inviscid and diffusive Oldroyd-B model was investigated by [T. M. Elgindi, F. Rousset, Commun. Pure Appl. Math.	 68 (2015), 2005--2021] where the global existence and uniqueness of the strong solution were established for arbitrarily large initial data. As pointed out by [A. V. Bhave, R. C. Armstrong, R. A. Brown, J. Chem. Phys., 95(1991), 2988-–3000], the diffusion coefficient is significantly smaller than other effects, it is interesting to study the non-diffusive model. In the present work, we obtain the global-in-time existence and uniqueness of the strong solution to the non-diffusive model with small initial data via deriving some uniform regularity estimates and taking vanishing diffusion limits. In addition, the large time behavior of the solution is studied and the optimal time-decay rates for each order of spatial derivatives are obtained. The main challenges focus on the lack of dissipation and regularity effects of the system and on the slower decay in the two-dimensional settings. A combination of the spectral analysis and the Fourier splitting method is adopted.
	\end{abstract}
	\vspace{4mm}
	
	
	\vspace{4mm}
	\section{Introduction}
	The interest for viscoelastic fluids has increased considerably, due to their connections with applied sciences. The motion of the fluids can be described by the Navier-Stokes equations coupling some constitutive laws of different types, see \cite{Bird_1, Bird_2} for more details. In this paper, we consider the following  {O}ldroyd-{B} type model in Eulerian coordinates:
	\begin{equation} \label{Oldroyd_B_d}
		\begin{cases}
			\partial_tu+(u\cdot\nabla) u+\nabla p=K\, {\rm div}\tau,\\
			\partial_t\tau+(u\cdot\nabla)\tau+\beta\tau=\alpha\mathbb{D}(u),\\
			{\rm div}\,u=0,	\\[2mm]
			(u,\tau)(x,0)=(u_0,\tau_0)(x),
		\end{cases}
	\end{equation}
	on $\mathbb{R}^2$ $\times$ $(0,\infty)$. (\ref{Oldroyd_B_d}) with a diffusion term $-\mu\Delta\tau$ on the left-hand side of the equation of $\tau$ was investigated by Elgindi and Rousset in \cite{Elgindi  Rousset 2015} where the global existence and uniqueness of the strong solution were established for arbitrarily large initial data. In this paper, we aim to study the global wellposedness and the large time behavior of the non-diffusive model (\ref{Oldroyd_B_d}).

	We will give an overview of study of the model. In fact, it is a simplified model of the following classical incompressible {O}ldroyd-{B} model\footnote{\eqref{Oldroyd_B_d} is the case that $\mu=0$, $\nu=0$ and $Q=0$.}:
	\begin{eqnarray} \label{Oldroyd_B}
		\begin{cases}
			\partial_tu+(u\cdot\nabla) u+\nabla p-\nu\Delta u=K \,{\rm div}\tau,\\
			\partial_t\tau+(u\cdot\nabla)\tau-\mu\Delta\tau+\beta\tau=Q(\nabla u,\tau)+\alpha\mathbb{D}(u),\\
			{\rm div}\,u=0,
		\end{cases}
	\end{eqnarray}
	where $u=u(x,t)$, $p=p(x,t)$, and $\tau=\tau(x,t)$ denote the velocity field of the fluid, the scalar pressure, and the tangential part of the stress tensor represented by a symmetric matrix, respectively. $\mathbb{D}(u)=\frac12(\nabla u+ \nabla u^T)$ is the symmetric part of the velocity gradient. The nonlinear term $Q(\nabla u,\tau)$ is a bilinear form:
	\begin{equation*}
		Q(\nabla u,\tau)=\Omega\tau - \tau\Omega + b(\mathbb{D}(u)\tau+\tau\mathbb{D}(u)).
	\end{equation*}
	$\Omega=\frac12(\nabla u- \nabla u^T)$ is the skew-symmetric part of velocity gradient and $b\in[-1,1]$. Those physical coefficients $\alpha,\beta ,\mu, \nu, K$ are constants that satisfy  $\alpha,\beta, K,\mu, \nu>0.$
	
	As pointed out by Bhave, Armstrong and Brown (\cite{Bhave 1991}), the diffusion coefficient  $\mu$ is significantly smaller than other effects. Thus some early works on the mathematical theory of the system \eqref{Oldroyd_B} focused on the non-diffusive case (i.e. $\nu>0,\mu = 0$ in \eqref{Oldroyd_B}). In this case, the model \eqref{Oldroyd_B} without diffusive term was first introduced by {O}ldroyd (\cite{Oldroyd 1958})  to describe the behavior of  viscoelastic fluids, which consists of both viscous and elastic components, and thus behaves as viscous fluid in some circumstances and as elastic solid in others. For the initial-boundary value problem, Guillop\'{e} and Saut(\cite{Guillo 1990}) established the local wellposedness of strong solutions in Sobolev space $H^s$  and obtained the global existence and uniqueness with small initial data and small coupling parameter $\alpha$. Later, Fern\'{a}ndez-Cara, Guill\'{e}n, and Ortega (\cite{Ortega 1998}) extended the result in the $L^p$ settings. Molinet and Talhouk (\cite{Molinet 2004}) proved that the results obtained by \cite{Guillo 1990} remain true without any restriction on the smallness of the coupling parameter. When considering the exterior domains, one needs to overcome the difficulty caused by both the boundary effect and unboundedness of the domain. Hieber, Naito, and Shibata (\cite{Hieber Naito 2012}) obtained the unique global strong solution with small initial data and small coupling parameter, see also \cite{Fang  Hieber Zi 2013} by Fang, Hieber, and Zi for the non-small coupling parameter case.  Chemin and Masmoudi (\cite{Chemin 2001}) studied the global wellposeness in the framework of critical Besov spaces and some blow-up criteria were also obtained. See also \cite{Chen Miao  2008, Zi 2014} for the case of the non-small coupling parameter in critical Besov spaces. Lions and Masmoudi (\cite{Lions 2000}) considered the case that $b=0$ and proved the existence of global weak solution for arbitrarily large initial data. In fact, for the case $b\neq0$, it is still open. For some studies of blow-up criteria, please refer to \cite{Lei 2010, Kupferman  2008}. Lei (\cite{Lei  2006}) obtained the global existence of classical solutions via the incompressible limit in periodic domains. Recently, Hieber, Wen, and Zi (\cite{Hieber Wen 2019}) studied the long time behaviors of the solutions in three dimensions and obtained the same decay rates as the heat equation, see also the extension by Huang, Wang, Wen, and Zi (\cite{Huang 2022}). For the case of infinite Weissenberg number, an energetic variational approach was first introduced by Lin, Liu, and Zhang (\cite{Lin-Liu-Zhang}) to understand the physical structure for the related systems (see for instance \cite{Hu-Lin,Hu-Wu,Lai,Lei1,Lei2,Lin} for more progress).
	
	For the diffusive model (i.e. $\mu > 0$ in \eqref{Oldroyd_B}),
	Constantin and Kliegl (\cite{Constantin 2012}) proved the global wellposedness of strong solutions for the two-dimensional Cauchy problem   with large initial data  and $\nu>0$. For the inviscid case, Elgindi and Rousset (\cite{Elgindi  Rousset 2015}) proved that the problem \eqref{Oldroyd_B} is global wellposed in $\mathbb{R}^2$ provided that the initial data is small enough. Later, Elgindi and Liu (\cite{Elgindi Liu 2015}) extended the results to the three-dimensional case. Very recently, Huang, Wang, Wen and Zi (\cite{Huang 2022}) obtained the optimal decay estimates with vanishing viscosity ($\nu\geq 0$) in three dimensions. When $\nu=0$, Deng, Luo and Yin (\cite{Yin_Deng}) obtained the global wellposedness of strong solutions and the $H^1$ time-decay rate as $(1+t)^{-\frac12}$ with small initial data in $\mathbb{R}^2$. When $\nu=0$ and $Q=0$, Elgindi and Rousset (\cite{Elgindi  Rousset 2015})
	established the global existence and uniqueness of strong  solutions  in $\mathbb{R}^2$.
	More precisely, they proved the following result with the diffusion coefficient $\mu>0$.
	\begin{pro}[Theorem 1.1, \cite{Elgindi  Rousset 2015}]\label{proposition_1}
		Assume that the initial data satisfy $(u_0,\tau_0)\in H^{s}(\mathbb{R}^2)$ with $\mathrm{div}\, u_0 = 0,\tau_0$ symmetric  and $s>2$, there exists a unique global solution $(u,\tau)\in C([0,\infty);H^{s}(\mathbb{R}^2))$  to the initial-value problem of \eqref{Oldroyd_B} with $\nu=0$ and $Q=0$.
	\end{pro}
	It is interesting to see whether the solution obtained in Proposition \ref{proposition_1} exists globally or not for the non-diffusive case.
	\subsection{Main results}
	Our aim in this paper is to investigate the global-in-time existence and uniqueness and the optimal time-decay rates of the solutions to the initial-value problem of \eqref{Oldroyd_B_d}. The first main result concerning the global existence and uniqueness is stated as follows.
	\begin{theorem}\label{wellposedness}
		Assume that $(u_0,\tau_0)\in H^3(\mathbb{R}^2)$ with $\mathrm{div} \,u_0 = 0$ and $\tau_0$ symmetric, then there exists a sufficiently small constant $\epsilon_0 >0$ such that
		the Cauchy problem \eqref{Oldroyd_B_d} admits a unique global solution $(u,\tau)\in L^\infty([0,\infty);H^3(\mathbb{R}^2))$ satisfying the following uniform regularity estimate:
		\begin{equation*}
			\|(u,\tau)(t)\|_{H^3}^2 + \int_0^t (\|\nabla u(s)\|_{H^2}^2 + \|\tau(s)\|_{H^3}^2 ){\rm d}s \leq C\|(u_0,\tau_0)\|_{H^3}^2,
		\end{equation*}
		provided that	$\|(u_0,\tau_0)\|_{H^3} \leq \epsilon_0.$
	\end{theorem}
	Based on the global existence and uniqueness of the solution, we get the second main result concerning the time-decay estimates.
	\begin{theorem}\label{thm_OB_d_decay}
		Under the assumptions of Theorem \ref{wellposedness}, assume in addition that $(u_0,\tau_0)\in L^1(\mathbb{R}^2)$, then the following  optimal time-decay estimates of the solution to the problem \eqref{Oldroyd_B_d} hold.
		\begin{enumerate}[i)]
			\item  Upper time-decay estimates of the solutions:
			\begin{eqnarray}\label{opti1}
				\ \|\nabla^ku(t)\|_{L^2}\le C (1+t)^{-\frac12-\frac{k}{2}},\ k=0,1,2,3,
			\end{eqnarray}
			and
			\begin{eqnarray}\label{opti2}
				\ \|\nabla^{k}\tau(t)\|_{L^2}\le C(1+t)^{-1-\frac{k}{2}},\ k=0,1,2,3,
			\end{eqnarray}
			for all $t>0$, where $C$ is a positive constant independent of time.
			
			\medskip
			
			\item In addition, assume that $\Big|\int_{\mathbb{R}^2}u_0(x){\rm d}x\Big| = c_2>0.$ Then there exists a positive time $t_1=t_1(\beta)$ such that
			\begin{eqnarray}\label{opti3}
				\|\nabla^ku(t)\|_{L^2}\ge \frac{1}{C} (1+t)^{-\frac12-\frac{k}{2}},\ k=0,1,2,3,
			\end{eqnarray}
			and
			\begin{eqnarray}\label{opti4}
				\|\nabla^{k}\tau(t)\|_{L^2}\ge \frac{1}{C} (1+t)^{-1-\frac{k}{2}},\ k=0,1,2,3,
			\end{eqnarray}
			for all $t\geq t_1$.
		\end{enumerate}
	\end{theorem}
	\begin{rem}
		For any $\mu>0$, Theorem \ref{thm_OB_d_decay} still holds for the system with diffusion.
	\end{rem}

	\subsection{Main ideas}
	In order to establish the global wellposedness result, we choose (\ref{Oldroyd_B_d}) with the diffusive term $-\mu\Delta\tau$ as an approximate system. To obtain the uniform regularity for $\mu,$ the diffusive term can not play much role. Instead we make full use of the damping term $\tau$. Combining with some compactness arguments, the unique global solution of the Cauchy problem for \eqref{Oldroyd_B_d} can be obtained via vanishing diffusion limit.
	
	To obtain some optimal time-decay estimates of the solution, the main challenges focus on deriving the sharp decay rate of the solution itself in $L^2$ norm due to the lower dimension. In fact, one can know from Lemma \ref{lemma_Greenfunction_7} that the time-decay rates of the low-frequency part of the solution to the linearized system \eqref{Oldroyd_B_d} will decrease as the dimension does. Our main strategy is to use spectral analysis together with energy method that the upper bound of the low frequency is a constant to get the sharp time-decay rates of the higher order derivatives of the solution. However, it seems not working for the lower order. More specifically, inspired by \cite{Dong 2006} where $$\|u(t)\|_{L^2}\le C (1+t)^{-\frac14},$$ can be derived by using an observation that $$(1+t)^{\frac{1}{2}}\| \nabla u(t)\|_{L^2}\longrightarrow 0 \,\,\,\,as\,\,\,\,t\rightarrow\infty,$$ and Lemma \ref{lemma_Greenfunction_8}, see Lemma \ref{lemma_upper_decay} for more details. To get the sharp time-decay rate $$\|u(t)\|_{L^2}\le C (1+t)^{-\frac12},$$ if we replace the upper bound of the low frequency $g_1^2(t)$ in \eqref{new_H1_L2_29} by a constant and use Lemma \ref{lemma_Greenfunction_7}, then an integral like $\int_0^t(1+t-s)^{-\frac12}(1+s)^{-1}{\rm d}s$ will turn up and it could not be dominated by $(1+t)^{-\frac12}$. The Fourier splitting method that the upper bound of the low frequency depends on a function of time can overcome this difficulty.
	
	%In order to guarantee the  time-decay rates is optimal in some sense, we need to make sure that the lower time-decay estimates is consistent with the upper time-decay estimates. Compared to \cite{Huang 2022}, the difficulty we have here is that the decay rates of the nonlinear term $(u\cdot\nabla) u$ and $(u\cdot\nabla) \tau$ are not high enough when the order is  high. Our strategy is to artificially increase derivatives  to the nonlinear term after the Fourier transform to increase the decay rates, see Lemma \ref{lower_bound} for details.	
	
	\medskip	
	
	The rest of the paper is organized as follows. In Section \ref{Section_2}  we prove the uniform regularity estimates for $\mu$ in \eqref{Oldroyd_B_1} and obtain Theorem \ref{approximate solution}. In Section \ref{Section_4} we use the vanishing diffusion limit technique to obtain the unique global solution of system \eqref{Oldroyd_B_d} and finish the proof of Theorem \ref{wellposedness}. In Section \ref{Section_5} we first analyze the linear part of the system \eqref{Oldroyd_B_d} and obtain the corresponding estimates of Green functions and the low-frequency part of the solution in \eqref{Oldroyd_B_d}, and then we obtain the optimal time-decay rates respectively for $u$ and $\tau$ and get Theorem \ref{thm_OB_d_decay}.

	Throughout the rest of the paper, let $C$ denote a generic positive constant depending on some known constants but independent of $\mu$, $\delta$, $t$, and $\eta_i$ for $i=1,2,3$.

	\section{Uniform regularity}\label{Section_2}
	
	To begin with, we use the following initial-value problem as an approximation of the problem \eqref{Oldroyd_B_d} as $\mu \to 0$, namely,
	\begin{eqnarray} \label{Oldroyd_B_1}
		\begin{cases}
			\partial_tu^\mu+(u^\mu\cdot\nabla) u^\mu+\nabla p^\mu=K \,{\rm div}\tau^\mu,\\
			\partial_t\tau^\mu+(u^\mu\cdot\nabla)\tau^\mu-\mu\Delta\tau^\mu+\beta\tau^\mu=\alpha\mathbb{D}(u^\mu),\\
			{\rm div}\,u^\mu=0, (u^\mu,\tau^\mu)(x,0) = (u_0,\tau_0).
		\end{cases}
	\end{eqnarray}
	The global wellposedness of problem \eqref{Oldroyd_B_1} for fixed $\mu>0$ was already stated in Proposition \ref{proposition_1}. In this section, we will establish the uniform regularity of the solutions to the problem \eqref{Oldroyd_B_1}, i.e.,
	\begin{theorem}\label{approximate solution}
		Suppose that $(u_0,\tau_0)\in H^3(\mathbb{R}^2)$ with $\mathrm{div} \,u_0 = 0$ and $\tau_0$ symmetric, then there exists a sufficiently small constant $\epsilon_0 >0$ independent
		of $\mu$ and $t$, such that the solutions to the Cauchy problem \eqref{Oldroyd_B_1}  satisfy the following uniform estimates:
		\begin{equation*}\label{uniform_estimates}
			\|(u^\mu,\tau^\mu)(t)\|_{H^3}^2 + \int_0^t (\|\nabla u^\mu(s)\|_{H^2}^2 + \|\tau^\mu(s)\|_{H^3}^2 + \mu\|\nabla\tau^\mu(s)\|_{H^3}^2){\rm d}s \leq C\|(u_0,\tau_0)\|_{H^3}^2,
		\end{equation*}
		for all $t>0$, provided that $\|(u_0,\tau_0)\|_{H^3} \leq \epsilon_0.$
	\end{theorem}

	\medskip
	For simplicity, we use $(u,\tau)$ to represent $(u^\mu,\tau^\mu)$. Before proving Theorem \ref{approximate solution}, we need some reformulations of the original system which are motivated by \cite{Zi 2014} and the references therein.	More specifically, applying the Leray projection operator $\mathbb{P}$ to the first equation of \eqref{Oldroyd_B_1} and the operator $\mathbb{P}{\rm div}\,$ to the second equation of \eqref{Oldroyd_B_1} respectively, we obtain that
	\begin{eqnarray} \label{OB_d_1}
		\begin{cases}
			\partial_tu+\mathbb{P}\left(u\cdot\nabla u\right)=K\, \mathbb{P}{\rm div}\tau,\\
			\partial_t\mathbb{P}{\rm div}\tau+\mathbb{P}{\rm div}\left(u\cdot\nabla\tau\right)-\mu\,\mathbb{P}{\rm div}\Delta\tau+\beta\,\mathbb{P}{\rm div}\tau=\frac{\alpha}{2}\Delta u.
		\end{cases}
	\end{eqnarray}
	Then, applying $\Lambda^{-1}=(\sqrt{-\Delta})^{-1}$ to (\ref{OB_d_1})$_2$ and denoting by
	\begin{eqnarray}\label{sigma}
		\sigma := \Lambda^{-1}\mathbb{P}{\rm div}\tau,
	\end{eqnarray} we can rewrite (\ref{OB_d_1}) as follows:
	\begin{eqnarray} \label{u_sigma_d}
		\begin{cases}
			\partial_tu-K \Lambda\sigma=\mathcal{F}_1,\\
			\partial_t\sigma -\mu\Delta\sigma +\beta\sigma+\frac{\alpha}{2}\Lambda u=\mathcal{F}_2,
		\end{cases}
	\end{eqnarray} where
	\begin{eqnarray*}
		\mathcal{F}_1=-\mathbb{P}\left(u\cdot\nabla u\right),\
		\mathcal{F}_2=-\Lambda^{-1}\mathbb{P}{\rm div}\left(u\cdot\nabla\tau\right).
	\end{eqnarray*}
	Here $\hat{\sigma}^j=i\left(\delta_{j,k}-\frac{\xi_j\xi_k}{|\xi|^2}\right)\frac{\xi_l}{|\xi|}\hat{\tau}^{l,k}$ where $\hat{f}$ denotes the Fourier transform of $f$.
	
	It is worth noticing that	for any $u\in L^2(\mathbb{R}^2)$, there holds
	\begin{equation*}
		\mathbb{P}(u)_i=u_i-\sum_{k=1}^{2}R_iR_k u_k,
	\end{equation*}
	where $R_iR_k=(-\Delta)^{-1}\partial_i\partial_k$.
	
	It is not difficult to get that
	\begin{equation}\label{bu_7}
		\|\mathbb{P}u\|_{L^2}^2\le C \|u\|_{L^2}^2.
	\end{equation}
	Combining \eqref{sigma} and \eqref{bu_7}, we can estimate $\sigma$ as follows:
	\begin{equation}\label{bu_8}
		\|\nabla^k\sigma\|_{L^2}^2\le C\|\nabla^k\tau\|_{L^2}^2,
	\end{equation}for $k = 0,1,2,3$.
	
	%\section{Proof of Theorem \ref{approximate solution}}\label{Section_3}

	The proof of Theorem \ref{approximate solution} highly relies on the following proposition.
	\begin{pro}\label{Prop2}
		Under the conditions of Theorem \ref{approximate solution}, there exist sufficiently small positive constants $\epsilon_0$ and $\delta$ independent of $\mu$ and $T$ such that
		if
		\begin{equation*}\label{apriori-assum}
			\sup_{0\leq s \leq T}\|(u,\tau)(s)\|_{H^3}\leq \delta,
		\end{equation*}
		for any given $T>0$, there holds
		\begin{equation*}\label{apriori-result}
			\sup_{0\leq s \leq T}\|(u,\tau)(s)\|_{H^3}\leq \frac{\delta}{2},
		\end{equation*}
		provided that $\|(u_0,\tau_0)\|_{H^3} \leq \epsilon_0.$
		
	\end{pro}
	
	\medskip
	
	The proof of Proposition \ref{Prop2} consists of the following Lemmas \ref{lemma_regularity_1}, \ref{lemma_regularity_2} and \ref{lemma_regularity_3}.
	\begin{lemma}\label{lemma_regularity_1}
		Under the assumptions of Proposition \ref{Prop2}, there exists a sufficiently small positive constant $\eta_1$ independent of $\mu, T$ such that
		\begin{equation}\label{est_H1}
			\begin{split}
				& \udt (\alpha\| u\|_{H^1}^2 + K\|\tau\|_{H^1}^2 + \eta_1\langle\Lambda u, \sigma\rangle) +  \frac{\beta K}{2}\|\tau\|_{H^1}^2 +\frac{\eta_1\alpha}{4}\|\Lambda u\|_{L^2}^2 + \mu K\|\nabla\tau\|_{H^1}^2
				\leq 0,
			\end{split}
		\end{equation}
		for all $0\leq t \leq T$.
	\end{lemma}
	\begin{proof}
		Multiplying (\ref{Oldroyd_B_1})$_1$ and (\ref{Oldroyd_B_1})$_2$ by $\alpha  u$ and $K \tau$, respectively, summing the results up, and using integration by parts, we have
		\begin{equation}\label{est_L2}
			\begin{split}
				&\frac12 \udt (\alpha\|u\|_{L^2}^2 + K\|\tau\|_{L^2}^2) +  \beta K\|\tau\|_{L^2}^2 + \mu K\|\nabla\tau\|_{L^2}^2	=  0.
			\end{split}
		\end{equation}
		Similarly, multiplying $\nabla$(\ref{Oldroyd_B_1})$_1$ and $\nabla$(\ref{Oldroyd_B_1})$_2$ by $\alpha\nabla u$ and $K  \nabla \tau$, respectively, we have
		\begin{equation*}
			\begin{split}
				&\frac12 \udt (\alpha\|\nabla u\|_{L^2}^2 + K\|\nabla\tau\|_{L^2}^2) +  \beta K\|\nabla\tau\|_{L^2}^2 + \mu K\|\nabla^2\tau\|_{L^2}^2\\
				=&  - \langle K \nabla(u\cdot\nabla\tau),\nabla\tau \rangle
				\,\le\, K\|\nabla u\|_{L^\infty}\|\nabla\tau\|_{L^2}^2
				\,\le\, C\delta K \|\nabla\tau\|_{L^2}^2.
			\end{split}
		\end{equation*}
		Letting $\delta\le \frac{\beta}{2C}$, then we obtain
		\begin{equation}\label{est_first}
			\begin{split}
				& \frac{1}{2}\udt (\alpha\|\nabla u\|_{L^2}^2 + K\|\nabla\tau\|_{L^2}^2) +  \frac{\beta K}{2}\|\nabla\tau\|_{L^2}^2 + \mu K\|\nabla^2\tau\|_{L^2}^2
				\leq 0.
			\end{split}
		\end{equation}	
	
		To derive the dissipative estimate of the velocity gradient, the equation of $\sigma$ plays an important role. More specifically, multiplying $\Lambda$(\ref{u_sigma_d})$_1$ and (\ref{u_sigma_d})$_2$ by $\sigma$ and $\Lambda u$, respectively, summing the results up, and using integration by parts, we have
		\begin{equation}\label{bu}	
			\begin{split}
				&\partial_t\langle\Lambda u, \sigma\rangle + \frac{\alpha}{2}\|\Lambda u\|_{L^2}^2\\
				= &\Big(K\|\Lambda \sigma\|_{L^2}^2 + \langle \mu\Delta\sigma,\Lambda u\rangle - \langle\beta\sigma,\Lambda u \rangle\Big)\\ &- \Big(\langle \Lambda\mathbb{P}(u\cdot \nabla u),\sigma\rangle + \langle \Lambda^{-1}\mathbb{P}\udiv(u\cdot \nabla \tau),\Lambda u\rangle\Big)\\
				=:&\,  I_1 - I_2.
			\end{split}
		\end{equation}
		For $I_1$ and $I_2$, using (\ref{bu_7}), we have that
		\begin{equation*}\label{bu_1}	
			\begin{split}
				|I_1|&\le K\|\Lambda \sigma\|_{L^2}^2 + \frac{\alpha}{16}\|\Lambda u\|_{L^2}^2 +  \frac{4\mu^2}{\alpha}\|\Delta\sigma\|_{L^2}^2 + \frac{\alpha}{16}\|\Lambda u\|_{L^2}^2 + \frac{4\beta^2}{\alpha}\|\sigma\|_{L^2}^2,\\
				|I_2|&\le \frac12\|\Lambda \sigma\|_{L^2}^2
				+ \frac12\|\mathbb{P}(u\cdot \nabla u)\|_{L^2}^2 +  \frac{\alpha}{16}\|\Lambda u\|_{L^2}^2  + \frac{4}{\alpha}\|\Lambda^{-1}\mathbb{P}\udiv(u\cdot \nabla \tau)\|_{L^2}^2\\
				&\le \frac12\|\Lambda \sigma\|_{L^2}^2 + C\|u\cdot \nabla u\|_{L^2}^2	 +  \frac{\alpha}{16}\|\Lambda u\|_{L^2}^2  + 	C\|u\cdot \nabla \tau\|_{L^2}^2.
			\end{split}
		\end{equation*}
		Then, substituting the above inequalities  into (\ref{bu}), we obtain that
		\begin{equation}\label{est_first_u}	
			\begin{split}
				&\partial_t\langle\Lambda u, \sigma\rangle + \frac{\alpha}{2}\|\Lambda u\|_{L^2}^2\\
				\leq & \,(\frac{3\alpha}{16} + C\delta^2)\|\Lambda u\|_{L^2}^2 + (K + \frac{4\beta^2}{\alpha} + \frac12)\|\sigma\|_{H^1}^2 + C\delta^2\|\tau\|_{H^1}^2 + C\mu^2\|\nabla^2\tau\|_{L^2}^2.
			\end{split}
		\end{equation}				
		Letting $\delta$ and $\eta_1>0$ small enough, then summing \eqref{est_L2}, \eqref{est_first} and $\eta_1$\eqref{est_first_u} up, and using \eqref{bu_8}, we get (\ref{est_H1}).
	\end{proof}
	
	\medskip
	
	In a similar way, we can obtain the following higher order estimates.
	\begin{lemma}\label{lemma_regularity_2}
		Under the assumptions of Proposition \ref{Prop2}, there exists a sufficiently small positive constant $\eta_2=\frac{\eta_1}{4}$ independent of $\mu, T$ such that
		\begin{equation}\label{est_H2}
			\begin{split}
				\udt (\alpha\| u\|_{H^2}^2 &+ K\|\tau\|_{H^2}^2 + \eta_1\langle\Lambda u, \sigma\rangle + \eta_2\langle\Lambda^2 u,\Lambda \sigma\rangle)\\ &+  \frac{\beta K}{4}\|\tau\|_{H^2}^2  + \frac{\eta_2\alpha}{8}\|\Lambda u\|_{H^1}^2 + \mu K\|\nabla\tau\|_{H^2}^2
				\leq  0,
			\end{split}
		\end{equation} for all $0\leq t \leq T$.
	\end{lemma}		
	\begin{proof}
		Multiplying $\nabla^2$(\ref{Oldroyd_B_1})$_1$ and $\nabla^2$ (\ref{Oldroyd_B_1})$_2$ by $\alpha  \nabla^2 u$ and $K  \nabla^2 \tau$, respectively, summing the results up, and using integration by parts, we have
		\begin{align}\label{est_second}
			\begin{split}
				&\frac12 \udt (\alpha\|\nabla^2 u\|_{L^2}^2 + K\|\nabla^2\tau\|_{L^2}^2) +  \beta K\|\nabla^2\tau\|_{L^2}^2 + \mu K\|\nabla^3\tau\|_{L^2}^2\\
				= & - \,\langle K \nabla^2(u\cdot\nabla\tau),\nabla^2\tau \rangle - \langle\alpha \nabla^2(u\cdot\nabla u),\nabla^2u \rangle\\
				\leq & \,C(\|\nabla \tau\|_{L^\infty}\|\nabla^2 u\|_{L^2}\|\nabla^2 \tau\|_{L^2} +  \|\nabla u\|_{L^\infty}\|\nabla^2 \tau\|_{L^2}^2 +  \|\nabla u\|_{L^\infty}\|\nabla^2 u\|_{L^2}^2)\\
				\leq &\, C\delta \|\nabla^2 u\|_{L^2}^2 + C\delta \|\nabla^2\tau\|_{L^2}^2.
			\end{split}
		\end{align}
		Multiplying $\Lambda^2$(\ref{u_sigma_d})$_1$ and  $\Lambda$(\ref{u_sigma_d})$_2$ by $\Lambda\sigma$ and  $\Lambda^2 u$, respectively, summing the results up, and using integration by parts, we have
		\begin{align}\label{bu_2}
			\begin{split}
				&\partial_t\langle\Lambda^2 u,\Lambda \sigma\rangle + \frac{\alpha}{2}\|\Lambda^2 u\|_{L^2}^2\\
				= & \,\Big(K\|\Lambda^2 \sigma\|_{L^2}^2 + \langle \mu\Lambda\,\Delta\sigma,\Lambda^2 u\rangle - \langle\beta\Lambda\sigma,\Lambda^2 u \rangle\Big)\\ &- \Big(\langle \Lambda^2\mathbb{P}(u\cdot \nabla u),\Lambda\sigma\rangle + \langle \mathbb{P}\udiv(u\cdot \nabla \tau),\Lambda^2 u\rangle\Big)\\
				=: & \,I_3 - I_4.
			\end{split}
		\end{align}
		For $I_3$ and $I_4$, using (\ref{bu_7}), we have that
		\begin{align}
			|I_3|&\le K\|\Lambda^2 \sigma\|_{L^2}^2 + \frac{\alpha}{16}\|\Lambda^2 u\|_{L^2}^2 + \frac{4\mu^2}{\alpha}\|\Lambda\,\Delta\sigma\|_{L^2}^2  + \frac{\alpha}{16}\|\Lambda^2 u\|_{L^2}^2 + \frac{4\beta^2}{\alpha}\|\Lambda\sigma\|_{L^2}^2,\label{I3}\\
			\nonumber |I_4|&\le \frac12\|\Lambda^2 \sigma\|_{L^2}^2
			+  \frac12\|\Lambda\mathbb{P}(u\cdot \nabla u)\|_{L^2}^2
			+  \frac{\alpha}{16}\|\Lambda^2 u\|_{L^2}^2 + \frac{4}{\alpha}\|\mathbb{P}\udiv(u\cdot \nabla \tau)\|_{L^2}^2\\
			&\le \frac12\|\Lambda^2 \sigma\|_{L^2}^2
			+  C\|\nabla(u\cdot \nabla u)\|_{L^2}^2 + \frac{\alpha}{16}\|\Lambda^2 u\|_{L^2}^2 + C\|\nabla(u\cdot \nabla \tau)\|_{L^2}^2\label{I4}.
		\end{align}
		Substituting (\ref{I3}) and (\ref{I4}) into (\ref{bu_2}), we get
		\begin{align}\label{est_second_u}
			\begin{split}
				&\partial_t\langle\Lambda^2 u,\Lambda \sigma\rangle + \frac{\alpha}{2}\|\Lambda^2 u\|_{L^2}^2\\
				\leq & \,(\frac{3\alpha}{16} + C\delta^2)\|\Lambda^2 u\|_{L^2}^2 + C\delta^2\|\nabla u\|_{L^2}^2\\
				&+ (K + \frac{4\beta^2}{\alpha} + \frac12)\|\nabla\sigma\|_{H^1}^2 + C\delta^2\|\nabla\tau\|_{H^1}^2 + C\mu^2\|\nabla^3\tau\|_{L^2}^2,
			\end{split}
		\end{align}
		where the facts that
		\begin{equation*}\label{bu_20}
			\begin{split}
				\|\nabla(u\cdot \nabla u)\|_{L^2}\le \| \nabla u\|_{L^\infty}\|\nabla u\|_{L^2} + \|  u\|_{L^\infty}\|\nabla^2 u\|_{L^2},\\
				\|\nabla(u\cdot \nabla \tau)\|_{L^2}\le \|\nabla u\|_{L^\infty}\|\nabla \tau\|_{L^2} + \| u\|_{L^\infty}\|\nabla^2 \tau\|_{L^2},
			\end{split}
		\end{equation*}
		are used.
		
		Letting $\eta_2= \frac14 \eta_1$  and $\delta$ small enough, then summing \eqref{est_H1}, \eqref{est_second} and $\eta_2$\eqref{est_second_u} up, and using \eqref{bu_8}, we get (\ref{est_H2}).
	\end{proof}
	
	Finally, we get the following uniform estimates up to the third order.
	\begin{lemma}\label{lemma_regularity_3}
		Under the assumptions of Proposition \ref{Prop2}, there holds
		\begin{equation}\label{eq_a_priori_est}
			\begin{cases}
				\|(u,\tau)(t)\|_{H^3}^2  \leq \frac{\delta}{2},\\[4mm]
				\displaystyle\int_0^t\left(\|\nabla u(s)\|_{H^2}^2 + \|\tau(s)\|_{H^3}^2 + \mu\|\nabla\tau(s)\|_{H^3}^2\right){\rm d}s\le C,
			\end{cases}
		\end{equation}for all $0\leq t \leq T.$
	\end{lemma}
	\begin{proof}
		Multiplying $\nabla^3$(\ref{Oldroyd_B_1})$_1$ and $\nabla^3$ (\ref{Oldroyd_B_1})$_2$ by $\alpha  \nabla^3 u$ and $K  \nabla^3 \tau$, respectively, summing the results up, and using integration by parts, we have
		\begin{equation}\label{est_third}
			\begin{split}
				&\frac12 \udt (\alpha\|\nabla^3 u\|_{L^2}^2 + K\|\nabla^3\tau\|_{L^2}^2) +  \beta K\|\nabla^3\tau\|_{L^2}^2 + \mu K\|\nabla^4\tau\|_{L^2}^2\\
				= & -\, \langle K \nabla^3(u\cdot\nabla\tau),\nabla^3\tau \rangle - \langle\alpha \nabla^3(u\cdot\nabla u),\nabla^3u \rangle\\
				\leq & \,C(\|\nabla u\|_{L^\infty}\|\nabla^3 \tau\|_{L^2}^2 + \|\nabla \tau\|_{L^\infty}\|\nabla^3 u\|_{L^2}\|\nabla^3 \tau\|_{L^2} + \|\nabla u\|_{L^\infty}\|\nabla^3 u\|_{L^2}^2  )\\
				\leq & \,C\delta \|\nabla^3 u\|_{L^2}^2 + C\delta\|\nabla^3\tau\|_{L^2}^2.
			\end{split}
		\end{equation}
		Similarly, multiplying $\Lambda^3$(\ref{u_sigma_d})$_1$ and $\Lambda^2$(\ref{u_sigma_d})$_2$ by $\Lambda^2\sigma$ and $\Lambda^3 u$, respectively, and using integration by parts, we have
		\begin{align}\label{bu_4}
			\begin{split}
				&\partial_t\langle\Lambda^3 u,\Lambda^2 \sigma\rangle + \frac{\alpha}{2}\|\Lambda^3 u\|_{L^2}^2\\
				= & \,\Big(K\|\Lambda^3 \sigma\|_{L^2}^2 + \langle \mu\Lambda^2\Delta\sigma,\Lambda^3 u\rangle   - \langle\beta\Lambda^2\sigma,\Lambda^3 u \rangle\Big)\\ &- \Big(\langle \Lambda^3\mathbb{P}(u\cdot \nabla u),\Lambda^2\sigma\rangle + \langle \Lambda\mathbb{P}\udiv(u\cdot \nabla \tau),\Lambda^3 u\rangle\Big)\\
				=: &\, I_5 - I_6.
			\end{split}
		\end{align}
		Using (\ref{bu_7}),	we can obtain the estimates of $I_5$ and $I_6$ as follows
		\begin{align}
			|I_5|&\le K\|\Lambda^3 \sigma\|_{L^2}^2 + \frac{\alpha}{16}\|\Lambda^3 u\|_{L^2}^2   + \frac{4\mu^2}{\alpha}\|\Lambda^2\Delta\sigma\|_{L^2}^2  + \frac{\alpha}{16}\|\Lambda^3 u\|_{L^2}^2 + \frac{4\beta^2}{\alpha}\|\Lambda^2\sigma\|_{L^2}^2,\label{I5}\\
			\nonumber|I_6|&\le \frac12\|\Lambda^3 \sigma\|_{L^2}^2
			+  \frac12\|\Lambda^2\mathbb{P}(u\cdot \nabla u)\|_{L^2}^2 +  \frac{\alpha}{16}\|\Lambda^3 u\|_{L^2}^2 + \frac{4}{\alpha}\|\Lambda\mathbb{P}\udiv(u\cdot \nabla \tau)\|_{L^2}^2\\
			&\le \frac12\|\Lambda^3 \sigma\|_{L^2}^2
			+  C\|\nabla^2(u\cdot \nabla u)\|_{L^2}^2 + \frac{\alpha}{16}\|\Lambda^3 u\|_{L^2}^2 + C\|\nabla^2(u\cdot \nabla \tau)\|_{L^2}^2.\label{I6}
		\end{align}
		Substituting (\ref{I5}) and (\ref{I6}) into (\ref{bu_4}), we can deduce that
		\begin{equation}\label{est_third_u}
			\begin{split}
				&\partial_t\langle\Lambda^3 u,\Lambda^2 \sigma\rangle + \frac{\alpha}{2}\|\Lambda^3 u\|_{L^2}^2\\
				\leq & \,(\frac{3\alpha}{16} + C\delta^2)\|\Lambda^3 u\|_{L^2}^2 + C\delta^2\|\nabla^2 u\|_{L^2}^2\\ &+ (K + \frac{4\beta^2}{\alpha} + \frac12)\|\nabla^2\sigma\|_{H^1}^2 + C\delta^2 \|\nabla^2\tau\|_{H^1}^2 + C\mu^2 \|\nabla^4\tau\|_{L^2}^2,
			\end{split}
		\end{equation}
		where the facts that
		\begin{equation}\label{bu_21}
			\begin{split}
				\|\nabla^2(u\cdot \nabla u)\|_{L^2}&\le \|  u\|_{L^\infty}\|\nabla^3 u\|_{L^2} + 3\| \nabla u\|_{L^\infty}\|\nabla^2 u\|_{L^2},\\
				\|\nabla^2(u\cdot \nabla \tau)\|_{L^2}&\le \| u\|_{L^\infty}\|\nabla^3 \tau\|_{L^2} + 2\|\nabla u\|_{L^\infty}\|\nabla^2 \tau\|_{L^2} + \|\nabla \tau\|_{L^\infty}\|\nabla^2 u\|_{L^2},
			\end{split}
		\end{equation}
		are used.
		
		Letting $\eta_3:= \frac14 \eta_2$  and $\delta$ small enough, summing \eqref{est_H2}, \eqref{est_third} and $\eta_3$\eqref{est_third_u} up, and using \eqref{bu_8}, we obtain that
		\begin{equation}\label{est_H3}
			\begin{split}
				\udt (\alpha\| u\|_{H^3}^2 &+ K\|\tau\|_{H^3}^2 + \sum_{i=1}^{3}\eta_i\langle\Lambda^i u,\Lambda^{i-1}\sigma\rangle)\\ &+  \frac{\beta K}{8}\|\tau\|_{H^3}^2  + \frac{\eta_3\alpha}{16}\|\Lambda u\|_{H^2}^2 + \mu K\|\nabla\tau\|_{H^3}^2
				\leq  0.
			\end{split}
		\end{equation}
		From the definition of $\eta_1$, $\eta_2$ and $\eta_3$, we have that
		\begin{equation}\label{bu_9}
			\begin{split}
				\frac12(\alpha\| u\|_{H^3}^2 + K\|\tau\|_{H^3}^2)&\le \alpha\| u\|_{H^3}^2 + K\|\tau\|_{H^3}^2 + \sum_{i=1}^{3}\eta_i\langle\Lambda^i u,\Lambda^{i-1}\sigma\rangle\\ &\le 2(\alpha\| u\|_{H^3}^2 + K\|\tau\|_{H^3}^2).
			\end{split}
		\end{equation}
		For all $0\leq t \leq T,$ integrating (\ref{est_H3}) over $[0, t]$  and utilizing   (\ref{bu_9}), we have that
		\begin{equation}\label{bu_10}
			\begin{split}
				&\frac12(\alpha\| u(t)\|_{H^3}^2 + K\|\tau(t)\|_{H^3}^2)  \\&+\int_0^t\left(\frac{\eta_3\alpha}{16}\|\nabla u(s)\|_{H^2}^2 + \frac{\beta K}{8}\|\tau(s)\|_{H^3}^2 + \mu K\|\nabla\tau(s)\|_{H^3}^2\right){\rm d}s\\
				\le &\,2(\alpha\| u_0\|_{H^3}^2 + K\|\tau_0\|_{H^3}^2)\,\le\,(2\alpha + 2K )\epsilon_0^2.
			\end{split}
		\end{equation}
		Letting
		\begin{equation*}\label{bu_11}
			\begin{split}
				\frac{4(\alpha+K)}{\min\{\alpha,K\}}\epsilon_0^2\le \frac{\delta}{2},
			\end{split}
		\end{equation*}
		we get (\ref{eq_a_priori_est})$_1$ from (\ref{bu_10}). Using (\ref{bu_10}) again, we get (\ref{eq_a_priori_est})$_2$ for some known positive constant $C$.
		
	\end{proof}
	
	\bigskip
	With Lemma \ref{lemma_regularity_3}, we finish the proof of Proposition \ref{Prop2}. Now we come to the proof of Theorem \ref{approximate solution} by using the standard continuity method.
	\subsection*{Proof of Theorem \ref{approximate solution}}
	For any fixed $\mu>0$, since
	\begin{equation*}
		\|(u_0,\tau_0)\|_{H^3}^2\le \epsilon_0^2\le \frac{\delta}{2},
	\end{equation*}
	and
	\begin{equation*}
		\|(u^\mu,\tau^\mu)(t)\|\in C([0,\infty);H^{3}(\mathbb{R}^2)),
	\end{equation*}
	there exists a time $T=T(\mu)>0$, such that
	\begin{equation}\label{conclusion}
		\|(u^\mu,\tau^\mu)(t)\|_{H^3}\leq \delta,
	\end{equation}	for all $t\in[0,T]$.
	
	Letting $T^*$ be the maximal life span such that (\ref{conclusion}) holds. In view of (\ref{conclusion}), it holds that $T^*>0$.
	Suppose that $T^*<+\infty$. Then, the continuity of the solution with respect to time yields that (\ref{conclusion}) holds on $[0,T^*]$, i.e.,
	\begin{equation}\label{continuity_method}
		\sup_{0\le s \le T^*}\|(u^\mu,\tau^\mu)(s)\|_{H^3}\le \delta.
	\end{equation}
	Then (\ref{continuity_method}) and Proposition \ref{Prop2} conclude that
	\begin{equation}\label{apriori-assum1}
		\|(u^\mu,\tau^\mu)(t)\|_{H^3}^2 \leq \frac{\delta}{2},
	\end{equation} for all $t\in[0,T^*].$
	Using (\ref{apriori-assum1}) and the continuity of the solution with respect to time again, we obtain that
	$T^*$ in (\ref{continuity_method}) can be replaced by $T^*+\sigma_0$ for a positive constant $\sigma_0$. This is a contradiction with the definition of $T^*$. Therefore $T^*$ must be $+\infty$.
	
	Hence for all $\mu>0$ and $t> 0$, there holds
	\begin{equation*}
		\begin{split}
			\|(u^\mu,\tau^\mu)(t)\|_{H^3}^2\le\delta.
		\end{split}
	\end{equation*}
	This together with (\ref{bu_10}) finishes the proof of Theorem \ref{approximate solution}.
	
	\section{Global existence and uniqueness}\label{Section_4}
	This section aims to  complete the proof of Theorem \ref{wellposedness}. By virtue of the uniform estimates  stated in Theorem \ref{approximate solution}, i.e.,
	\begin{equation*}
		\|(u^\mu,\tau^\mu)(t)\|_{H^3}^2 + \int_0^t (\|\nabla u^\mu(s)\|_{H^2}^2 + \|\tau^\mu(s)\|_{H^3}^2 + \mu\|\nabla\tau^\mu(s)\|_{H^3}^2){\rm d}s \leq C\|(u_0,\tau_0)\|_{H^3}^2.
	\end{equation*}
	Combining the above inequality with the equation \eqref{Oldroyd_B_1}, we can easily obtain that
	\begin{equation*}
		\|(\partial_t u^\mu,\partial_t \tau^\mu)(t)\|_{H^2}^2\le C.
	\end{equation*}

	By virtue of some standard weak (or weak*) convergence results and the Aubin-Lions Lemma (see for instance \cite{Simon 1987}), there exists a $(u,\tau)\in L^\infty([0,\infty);H^3(\mathbb{R}^2))$ which is a limit of $(u^\mu,\tau^\mu)$ (take subsequence if necessary) in some sense and solves (\ref{approximate solution}).
	
	For the uniqueness, we suppose that there are two pairs of the solutions $(u_1,\tau_1)$ and $(u_2,\tau_2)$. Denote $w=u_1-u_2$, $v=\tau_1-\tau_2$ satisfying
	\begin{equation} \label{convergence_6}
		\begin{cases}
			\partial_tw+(w\cdot\nabla) u_1 +u_2\cdot\nabla w +\nabla (p_1-p_2)=K\, {\rm div}\,v,\\
			\partial_tv+(w\cdot\nabla)\tau_1+u_2\cdot\nabla v+\beta v=\alpha\mathbb{D}(w).
		\end{cases}
	\end{equation}
	Multiplying (\ref{convergence_6})$_1$ and (\ref{convergence_6})$_2$ by $\alpha  w$ and $K v$, respectively, summing the results up, and using integration by parts, we have
	\begin{equation*}
		\begin{split}
			&\frac12 \udt (\alpha\|w\|_{L^2}^2 + K\|v\|_{L^2}^2) +  \beta K\|v\|_{L^2}^2\\
			\le&-\alpha\langle w\cdot\nabla u_1,w\rangle- \alpha\langle u_2\cdot\nabla w,w\rangle-K\langle w\cdot\nabla\tau_1,v\rangle-K\langle u_2\cdot\nabla v,v\rangle\\
			\le&\,C(\alpha\|w\|_{L^2}^2 + K\|v\|_{L^2}^2),
		\end{split}
	\end{equation*}
	which implies
	\begin{equation*}
		\alpha\|w(t)\|_{L^2}^2 + K\|v(t)\|_{L^2}^2\le e^{Ct}(\alpha\|w(0)\|_{L^2}^2 + K\|v(0)\|_{L^2}^2)=0.
	\end{equation*}
	Thus, $w=u_1-u_2=0$ and $v=\tau_1-\tau_2=0$. The proof of the uniqueness is complete.
	
	\section{Decay estimates for the nonlinear system}\label{Section_5}
	In this section, we will establish the upper and lower decay estimates to the solutions of the Cauchy problem \eqref{Oldroyd_B_d} and finish the  proof of Theorem \ref{thm_OB_d_decay}. We consider $\mu=0$ in \eqref{u_sigma_d}:
	\begin{eqnarray} \label{u_sigma_1}
		\begin{cases}
			\partial_tu-K \Lambda\sigma=\mathcal{F}_1,\\
			\partial_t\sigma +\beta\sigma+\frac{\alpha}{2}\Lambda u=\mathcal{F}_2,
		\end{cases}
	\end{eqnarray} where
	\begin{eqnarray*}
		\mathcal{F}_1=-\mathbb{P}\left(u\cdot\nabla u\right),\
		\mathcal{F}_2=-\Lambda^{-1}\mathbb{P}{\rm div}\left(u\cdot\nabla\tau\right).
	\end{eqnarray*}
	
	\subsection{Some estimates of the low-frequency parts}
	
	Consider the linear part of the system \eqref{u_sigma_1}, i.e.,
	\begin{equation} \label{Greenfunction_1}
		\begin{cases}
			\partial_tu-K \Lambda\sigma=0,\\
			\partial_t\sigma +\beta\sigma+\frac{\alpha}{2}\Lambda u=0.
		\end{cases}
	\end{equation}
	Note that the 3{D} case of \eqref{Greenfunction_1} with viscosity and diffusion has already been analyzed by Huang, the second author, the third author, and Zi (\cite{Huang 2022}) (see Lemmas 2.1, 4.1, 4.5 and 4.6 and Proposition 2.3 therein). After some slight modifications, we can get the similar results in the 2{D} case.
	
	To begin with, applying Fourier transform to system \eqref{Greenfunction_1}, we get that
	\begin{equation} \label{Greenfunction_2}
		\begin{cases}
			\partial_t\hat{u}^j-K|\xi|\hat{\sigma}^j=0,\\
			\partial_t\hat{\sigma}^j+\beta\hat{\sigma}^j+\frac{\alpha}{2}|\xi| \hat{u}^j=0.
		\end{cases}
	\end{equation}
	\begin{lemma}\label{lemma_Greenfunction_1}
		(Lemma 2.1, \cite{Huang 2022} for the case $\mu, \varepsilon=0$) The system \eqref{Greenfunction_2} can be solved as follows  :
		\begin{equation*}
			\begin{cases}
				\hat{u} = \mathcal{G}_3 \hat{u}_0 + K|\xi|\mathcal{G}_1\hat{\sigma}_0,\\
				\hat{\sigma} = -\frac{\alpha}{2}|\xi|\mathcal{G}_1 \hat{u}_0 + \mathcal{G}_2\hat{\sigma}_0,
			\end{cases}
		\end{equation*}
		where
		\begin{equation}\label{lemma_Greenfunction_3+1}
			\begin{split}
				\mathcal{G}_1(\xi,t)=\frac{e^{\lambda_+t}-e^{\lambda_-t}}{\lambda_+-\lambda_-}, \ \mathcal{G}_2(\xi,t)&=\frac{\lambda_+e^{\lambda_+t}-\lambda_-e^{\lambda_-t}}{\lambda_+-\lambda_-}, \\ \mathcal{G}_3(\xi,t)=\frac{\lambda_+e^{\lambda_-t}-\lambda_-e^{\lambda_+t}}{\lambda_+-\lambda_-},\,
				\lambda_{\pm}&=\frac{-\beta\pm\sqrt{\beta^2-2\alpha K|\xi|^2}}{2}.
			\end{split}
		\end{equation}
	\end{lemma}
	Due to the explicit expression of the solution, we can easily get some estimates as follows.
	\begin{lemma}\label{lemma_Greenfunction_4}
		There exist  positive constants $R=R(\alpha,\beta,K)$,  $\theta=\theta(\alpha,\beta,K)$ and $C=C(\alpha,\beta,K)$, such that, for any $|\xi|\leq R$ and $t>0$, it holds that
		\begin{equation*}\label{lemma_Greenfunction_5}
			\begin{split}
				&\left|\mathcal{G}_1(\xi,t)\right|,\left|\mathcal{G}_3(\xi,t)\right|\leq Ce^{-\theta|\xi|^2t},\\
				&|\mathcal{G}_2(\xi,t)|
				\leq C\left(|\xi|^2 e^{-\theta|\xi|^2t} + e^{-\frac{\beta t}{2}}\right).
			\end{split}
		\end{equation*}
	\end{lemma}
	\begin{rem}
		The proof of Lemma \ref{lemma_Greenfunction_4} is similar to the proof of Proposition  $2.3$ and Lemma $4.5$ in \cite{Huang 2022} with only the difference of dimension.
	\end{rem}
	Consequently, we can obtain the upper bound of the low-frequency part of the solution satisfying \eqref{u_sigma_1}.
	\begin{lemma} \label{lemma_Greenfunction_7}
		Assume that $(u_0,\tau_0)\in L^1(\mathbb{R}^2)$, it holds for the solution to \eqref{u_sigma_1} that
		\begin{equation*} \label{lemma_Greenfunction_8}
			\begin{split}
				\left(\int_{|\xi|\leq R}
				|\xi|^{2k}|\hat{u}(t)|^2
				{\rm d}\xi\right)^\frac{1}{2}
				\le&\,C(1+t)^{-\frac12-\frac{k}{2}} + C\int_0^t	(1+t-s)^{-\frac12-\frac{k}{2}}(\|\hat{\mathcal{F}}_1 \|_{L^\infty} + \|\hat{\mathcal{F}}_2 \|_{L^\infty})
				{\rm d}s,\\
				\left(\int_{|\xi|\leq R}
				|\xi|^{2k}|\hat{\sigma}(t)|^2
				{\rm d}\xi\right)^\frac{1}{2}
				\le &\,C(1+t)^{-1-\frac{k}{2}} + C\int_0^t	(1+t-s)^{-1-\frac{k}{2}}(\|\hat{\mathcal{F}}_1 \|_{L^\infty} + \|\hat{\mathcal{F}}_2 \|_{L^\infty})
				{\rm d}s.
			\end{split}
		\end{equation*}
	\end{lemma}
	\begin{proof}
		From the Duhamel's principle, we have that
		\begin{align}
			\label{Greenfunction_13-1}	\hat{u}(t) =& \,\mathcal{G}_3 \hat{u}_0 + K|\xi|\mathcal{G}_1\hat{\sigma}_0 + \int_0^t\mathcal{G}_3(t-s)\hat{\mathcal{F}}_1 (s) + K|\xi|\mathcal{G}_1(t-s)\hat{\mathcal{F}}_2 (s)
			{\rm d}s,\\
			\label{Greenfunction_13-2}	\hat{\sigma}(t) =& -\frac{\alpha}{2}|\xi|\mathcal{G}_1 \hat{u}_0 + \mathcal{G}_2\hat{\sigma}_0 + \int_0^t-\frac{\alpha}{2}|\xi|\mathcal{G}_1(t-s)\hat{\mathcal{F}}_1 (s) + \mathcal{G}_2(t-s)\hat{\mathcal{F}}_2 (s){\rm d}s.
		\end{align}
		It follows from Lemma \ref{lemma_Greenfunction_4} and Minkowski's  inequality that
		\begin{align*}
			&\left(\int_{|\xi|\leq R}|\xi|^{2k}|\hat{u}(t)|^2{\rm d}\xi\right)^{\frac12}\\
			\le & ~C(\|\hat{u}_0\|_{L^\infty} + \|\hat{\tau}_0\|_{L^\infty})\left(\int_{|\xi|\leq R}|\xi|^{2k}e^{-2\theta|\xi|^2t}{\rm d}\xi\right)^{\frac12}\\
			&+ ~C\left(\int_{|\xi|\leq R}|\xi|^{2k}\Big|\int_0^t\mathcal{G}_3(t-s)\hat{\mathcal{F}}_1 (\xi,s) + K|\xi|\mathcal{G}_1(t-s)\hat{\mathcal{F}}_2 (\xi,s){\rm d}s\Big|^2{\rm d}\xi\right)^\frac12\\
			\le & ~C(1+t)^{-\frac12-\frac{k}{2}} + C\int_0^t\Big(\int_{|\xi|\leq R}|\xi|^{2k}e^{-2\theta|\xi|^2(t-s)}(|\hat{\mathcal{F}}_1 (\xi,s)|^2 + |\hat{\mathcal{F}}_2 (\xi,s)|^2){\rm d}\xi\Big)^\frac12{\rm d}s\\
			\le &~C(1+t)^{-\frac12-\frac{k}{2}} + C\int_0^t(1+t-s)^{-\frac12-\frac{k}{2}}(\|\hat{\mathcal{F}}_1 (\cdot,s)\|_{L^\infty} + \|\hat{\mathcal{F}}_2 (\cdot,s)\|_{L^\infty}){\rm d}s.
		\end{align*}
	
		By similar calculations, we obtain that
		\begin{align*}
			&\left(\int_{|\xi|\leq R}
			|\xi|^{2k}|\hat{\sigma}(t)|^2{\rm d}\xi\right)^\frac{1}{2}\\
			\le& ~C\|\hat{u}_0\|_{L^\infty}\left(\int_{|\xi|\leq R}|\xi|^{2k+2}e^{-2\theta|\xi|^2t}{\rm d}\xi\right)^{\frac12} + C\|\hat{\tau}_0\|_{L^\infty}\left(\int_{|\xi|\leq R}  |\xi|^{2k+4}e^{-2\theta|\xi|^2 t}+|\xi|^{2k}e^{-\beta t}{\rm d}\xi \right)^\frac12\\
			&+ ~C\left(\int_{|\xi|\leq R}|\xi|^{2k}\Big|\int_0^t-\frac{\alpha}{2}|\xi|\mathcal{G}_1(t-s)\hat{\mathcal{F}}_1 (\xi,s) + \mathcal{G}_2(t-s)\hat{\mathcal{F}}_2 (\xi,s){\rm d}s\Big|^2{\rm d}\xi\right)^\frac12\\
			\le& C(1+t)^{-1-\frac{k}{2}} + C\int_0^t	(1+t-s)^{-1-\frac{k}{2}}(\|\hat{\mathcal{F}}_1(\cdot,s)\|_{L^\infty} + \|\hat{\mathcal{F}}_2(\cdot,s) \|_{L^\infty})
			{\rm d}s.
		\end{align*}
	The proof of Lemma \ref{lemma_Greenfunction_7} is complete.
	\end{proof}
	Next, we consider the lower bound estimates of $\mathcal{G}_1(\xi,t)$ and $\mathcal{G}_3(\xi,t)$.
	\begin{lemma}\label{lemma_Greenfunction_10}
		Let  $R$ be the constant chosen in Lemma \ref{lemma_Greenfunction_4}. There exist three positive constants $\eta=\eta(\alpha,\beta,K)$, $C=C(\alpha,\beta,K)$ and  $t_1=t_1(\beta)$, such that
		\begin{equation}\label{Greenfunction_16}
			|\mathcal{G}_1(\xi,t)| \geq \frac{1}{C} e^{-\eta |\xi|^2 t},~~|\mathcal{G}_3(\xi,t)| \geq \frac{1}{C} e^{-\eta |\xi|^2 t}, \ {\rm for}\ {\rm all}\ |\xi| \leq R \ {\rm and}\ t\geq t_1.
		\end{equation}
	\end{lemma}
	\begin{proof}
		From Lemma \ref{lemma_Greenfunction_4}, there holds
		\begin{equation} \label{Greenfunction_17}
			\frac{\sqrt{2}}{2}\beta\leq \lambda_+ - \lambda_-=\sqrt{\beta^2-2\alpha\kappa|\xi|^2}\leq \beta,
		\end{equation} for all $|\xi|\leq R$, where $R$ is sufficiently small.
		
		Noticing that
		\begin{equation*}\label{Greenfunction_18}
			\lambda_+ =\frac{-\alpha K|\xi|^2}{\beta+\sqrt{\beta^2-2\alpha K|\xi|^2}} \geq - \frac{\alpha K}{\beta}|\xi|^2 =:-\eta|\xi|^2,
		\end{equation*}
		there exists a time $t_1 = \frac{\sqrt{2}\ln 2}{\beta},$ such that
		\begin{equation} \label{Greenfunction_19}
			|e^{\lambda_+t}-e^{\lambda_-t}| = \left|e^{\lambda_+ t}\big(1 - e^ {-(\lambda_+-\lambda_-)t}\big)\right|   \geq  \frac{1}{2}e^{-\eta |\xi|^2 t},
		\end{equation}
		and
		\begin{equation} \label{Greenfunction_20}
			|\lambda_+e^{\lambda_-t}-\lambda_-e^{\lambda_+t}| = |e^{\lambda_+ t}\big(\lambda_+ e^ {-(\lambda_+-\lambda_-)t} - \lambda_-\big)| \geq (\lambda_+-\lambda_-)e^{-\eta |\xi|^2t},
		\end{equation}  for any $t\geq t_1$.
	
		Then (\ref{lemma_Greenfunction_3+1}) combined with (\ref{Greenfunction_17}), (\ref{Greenfunction_19}), and (\ref{Greenfunction_20}) yields (\ref{Greenfunction_16}). Hence we finish the proof of Lemma \ref{lemma_Greenfunction_10}.
	\end{proof}

	With Lemmas \ref{lemma_Greenfunction_4} and \ref{lemma_Greenfunction_10}, the lower bounds of the linear part of the solution can be estimated as follows.
	\begin{lemma}\label{lemma_Greenfunction_11}
		Under the assumptions of Lemma \ref{lemma_Greenfunction_10}, and in addition that $\Big|\int_{\mathbb{R}^2}u_0(x) {\rm d}x\Big| = c_2>0,$  there exists a positive generic constant $C = C(\alpha,\beta,K,c_2,\|\tau_0\|_{L^1})$, such that
		\begin{equation}\label{lemma_Greenfunction_12}
			\begin{split}
				\||\xi|^k\left(\mathcal{G}_3 \hat{u}_0 + K|\xi|\mathcal{G}_1\hat{\sigma}_0\right)\|_{L^2} \geq  \frac{1}{C}(1 + t)^{-\frac12-\frac{k}{2}},
			\end{split}
		\end{equation}
		\begin{equation}\label{lemma_Greenfunction_13}
			\,\,\,\,\,\||\xi|^k\left(-\frac{\alpha}{2}|\xi|\mathcal{G}_1 \hat{u}_0 + \mathcal{G}_2\hat{\sigma}_0\right)\|_{L^2} \geq \frac{1}{C}(1 + t)^{-1-\frac{k}{2}},
		\end{equation}for all $t\geq t_1$ and $k=0, 1, 2, 3 $.
	\end{lemma}
	\begin{proof}
		First of all, since $u_0\in L^1(\mathbb{R}^2),$ then $\hat{u_0}\in C(\mathbb{R}^2).$ There exsits a constant $R'>0$, such that
		\begin{equation*}
			|\hat{u_0}(\xi)|\ge \frac{c_2}{2},\text{ for all } 0\le|\xi|\leq R'.
		\end{equation*}	
		For simplicity, we assume $R'\le R$,
		then we have
		\begin{equation}\label{Greenfunction_21}
			\begin{split}
				&\||\xi|^k\left(\mathcal{G}_3 \hat{u}_0 + K|\xi|\mathcal{G}_1\hat{\sigma}_0\right)\|_{L^2}\\
				=&
				\left\||\xi|^k\mathcal{G}_3(\xi,t)\hat{u}_0(\xi) +  	K|\xi|^{k+1}\mathcal{G}_1(\xi,t)\hat{\sigma}_0(\xi)\right\|_{L^2}\\
				\geq & \left(\int_{|\xi|\leq R'}|\xi|^{2k}\big|\mathcal{G}_3(\xi,t)\hat{u}_0(\xi) \big|^2 {\rm d}\xi\right)^\frac12 -\left(\int_{|\xi|\leq R'} K^2|\xi|^{2k+2}|\mathcal{G}_1(\xi,t)\hat{\sigma}_0(\xi)|^2 {\rm d}\xi\right)^\frac12
				\\ =:&\,
				K_1-K_2.
			\end{split}
		\end{equation}
		From Lemma \ref{lemma_Greenfunction_10}, we have that
		\begin{equation}\label{Greenfunction_22}
			\begin{split}
				K_1 \geq \frac{1}{C}\left(\int_{|\xi|\leq R'}|\xi|^{2k}e^{-2\eta |\xi|^2t}{\rm d} \xi\right)^\frac12
				\ge \frac{1}{C}(1+t)^{-\frac12-\frac{k}{2}},
			\end{split}
		\end{equation}	for all $t\geq t_1$.
		On the other hand,  Lemma \ref{lemma_Greenfunction_4} yields
		\begin{equation}\label{Greenfunction_23}
			\begin{split}
				K_2 \leq & ~C\|\hat{\sigma}_0\|_{L^\infty}\left( \int_{|\xi|\leq R'}   |\xi|^{2k+2} e^{-2\theta|\xi|^2t} {\rm d} \xi\right)^\frac12
				\le
				C(1+t)^{-1-\frac{k}{2}}.
			\end{split}
		\end{equation}
		(\ref{Greenfunction_21}), (\ref{Greenfunction_22}), and (\ref{Greenfunction_23}) yield (\ref{lemma_Greenfunction_12}) for all $t\ge t_1$.
		
		Next, notice that
		\begin{equation}\label{Greenfunction_24}
			\begin{split}
				&\||\xi|^k\left(-\frac{\alpha}{2}|\xi|\mathcal{G}_1 \hat{u}_0 + \mathcal{G}_2\hat{\sigma}_0\right)\|_{L^2}
				\\=  & \left\|-\frac{\alpha}{2}|\xi|^{k+1}\mathcal{G}_1(\xi,t)\hat{u}_0(\xi) +  |\xi|^k \mathcal{G}_2(\xi,t)\hat{\sigma}_0(\xi)\right\|_{L^2}\\
				\ge&\frac{\alpha}{2}\left(\int_{|\xi|\leq R'}|\xi|^{2k+2}|\mathcal{G}_1(\xi,t)|^2|\hat{u}_0(\xi)|^2 {\rm d}\xi\right)^\frac12
				- \left(\int_{|\xi|\leq R'}|\xi|^{2k}|\mathcal{G}_2(\xi,t)|^2|\hat{\sigma}_0(\xi)|^2 {\rm d}\xi\right)^\frac12\\
				=:&\, K_3 - K_4.
			\end{split}
		\end{equation}
		Similar to the analysis of $K_1$ and $K_2$, we have
		\begin{equation}\label{Greenfunction_25}
			K_3 \ge \frac{1}{C}(1 + t)^{-1-\frac{k}{2}},
		\end{equation}
		and
		\begin{equation}\label{Greenfunction_26}
			\begin{split}
				K_4 \le& \|\hat{\sigma}_0\|_{L^\infty} \left(\int_{|\xi|\leq R'} |\xi|^{2k}|\mathcal{G}_2(\xi,t)|^2{\rm d} \xi\right)^\frac12
				\\ \le  & C\left(\int_{|\xi|\leq R'}  |\xi|^{4+2k}e^{-2\theta|\xi|^2 t}+|\xi|^{2k}e^{-\beta t}{\rm d}\xi \right)^\frac12\\
				\le  & \,C(1 + t)^{-\frac{3}{2}-\frac{k}{2}},
			\end{split}
		\end{equation} for all $t\geq t_1$.

		It follows from \eqref{Greenfunction_24}, (\ref{Greenfunction_25}), and \eqref{Greenfunction_26} that  \eqref{lemma_Greenfunction_13} holds for all $t\ge t_1$. Therefore letting $t\ge t_1$, we finish the proof of Lemma \ref{lemma_Greenfunction_11}.
	\end{proof}
	
	\subsection{Upper time-decay estimates}To begin with, we define that
	\begin{equation*}
		\begin{split}
			\mathcal{H}_1(t) & :=  \alpha\| u\|_{H^1}^2 + K\| \tau\|_{H^1}^2 + \eta_1\langle\Lambda u,\sigma\rangle = O(\|(u,\tau)\|_{H^1}),\\
			\mathcal{H}_2(t) & :=  \alpha\|\nabla  u\|_{H^1}^2 + K\|\nabla \tau\|_{H^1}^2 + \eta_2\langle\Lambda^2 u,\Lambda\sigma\rangle = O(\|\nabla(u,\tau)\|_{H^1}),\\
			\mathcal{H}_3(t) & := \alpha\|\nabla^2  u\|_{H^1}^2 + K\|\nabla^2 \tau\|_{H^1}^2 + \eta_3\langle\Lambda^3 u,\Lambda^2\sigma\rangle = O(\|\nabla^2(u,\tau)\|_{H^1}).
		\end{split}
	\end{equation*}
	\begin{lemma}\label{lemma_upper_decay}
		Under the assumptions of Theorem \ref{thm_OB_d_decay}, we have
		\begin{equation*}\label{utauH1-1}
			\| u(t)\|_{H^1}^2 + \| \tau(t)\|_{H^1}^2\le C(1+t)^{-\frac{1}{2}},
		\end{equation*} for all $t>0.$
	\end{lemma}
	\begin{proof}
		Recalling from \eqref{est_first} ($\mu=0$) that
		\begin{equation*}
			\begin{split}
				\udt (\alpha\|\nabla u(t)\|_{L^2}^2 + K\|\nabla\tau(t)\|_{L^2}^2) \leq 0,
			\end{split}
		\end{equation*}
		then we have
		\begin{equation}\label{nau}
			\begin{split}
				\alpha\|\nabla u(t)\|_{L^2}^2 + K\|\nabla\tau(t)\|_{L^2}^2\le \alpha\|\nabla u(s)\|_{L^2}^2 + K\|\nabla\tau(s)\|_{L^2}^2,
			\end{split}
		\end{equation}for $t\ge s\ge 0$.
		
		By virtue of \eqref{uniform_estimates} with $\mu=0$, there holds
		\begin{equation*}
			\int_0^{+\infty} (\|\nabla u(s)\|_{H^2}^2 + \|\tau(s)\|_{H^3}^2){\rm d}s \leq C.
		\end{equation*}
		This combined with (\ref{nau}) yields
		\begin{equation*}\label{new_H1_L2_7}
			\begin{split}
				\frac{t}{2}\alpha\|\nabla u(t)\|_{L^2}^2 + \frac{t}{2}K\|\nabla\tau(t)\|_{L^2}^2\le \int_{\frac{t}{2}}^t (\alpha\|\nabla u(s)\|_{L^2}^2 + K\|\nabla\tau(s)\|_{L^2}^2){\rm d}s\longrightarrow 0\,\,\,\,\text{as}\,\,\,\,t\rightarrow+\infty.
			\end{split}
		\end{equation*}
		Namely, we have
		\begin{equation}\label{new_H1_L2_19}
			\begin{split}
				\varphi(t):=(1+t)^{\frac12}\|\nabla u(t)\|_{L^2}\longrightarrow 0 \,\,\,\,\text{and}\,\,\,\, \psi(t):=(1+t)^{\frac12}\|\nabla \tau(t)\|_{L^2}\longrightarrow 0 \,\,\,\,\text{as}\,\,\,\,t\rightarrow\infty.
			\end{split}
		\end{equation}
	
		Next, from \eqref{est_H1} ($\mu=0$), we have
		\begin{equation}\label{new_H1_L2}
			\begin{split}
				\udt \mathcal{H}_1(t) +  \frac{\beta K}{2}\|\tau\|_{H^1}^2 +\frac{\eta_1\alpha}{4}\|\Lambda u\|_{L^2}^2
				\leq 0.
			\end{split}
		\end{equation}
		Noticing that
		\begin{equation*}\label{notice}
			\begin{split}
				\|\Lambda u\|_{L^2}^2 = \|\nabla u\|_{L^2}^2 &=  \frac12 \|\nabla u\|_{L^2}^2 + \frac12 \int_{|\xi|\ge R_{1}}
				|\xi|^2|\hat{u}|^2	{\rm d}\xi + \frac12 \int_{|\xi|\le R_{1}}
				|\xi|^2|\hat{u}|^2	{\rm d}\xi\\
				&\ge \frac12 \|\nabla u\|_{L^2}^2 + \frac12 R_{1}^{2}\int_{|\xi|\ge R_{1}}
				|\hat{u}|^2	{\rm d}\xi,
			\end{split}
		\end{equation*}
		where $R_{1}:=\min\{1, R\}$. Without loss of generality, we assume that  $\frac{\eta_1}{8}\le \frac{\beta}{2}$, then \eqref{new_H1_L2} can be rewritten as
		\begin{equation}\label{dtH1}
			\begin{split}
				\udt \mathcal{H}_1(t) + \frac{\eta_1 R_{1}^{2}}{16}(2\alpha\| u\|_{H^1}^2 + 2K\| \tau\|_{H^1}^2)\le \frac{\eta_1\alpha}{8}\int_{|\xi|\le R_{1}}
				|\hat{u}|^2	{\rm d}\xi.
			\end{split}
		\end{equation}
		Substituting \eqref{bu_9} into (\ref{dtH1}), we obtain
		\begin{equation} \label{new1}
			\begin{split}
				\udt \mathcal{H}_1(t) + \frac{\eta_1 R_{1}^{2}}{16}\mathcal{H}_1(t)
				\le \frac{\eta_1\alpha}{8}\int_{|\xi|\le R_{1}}
				|\hat{u}|^2	{\rm d}\xi.
			\end{split}
		\end{equation}
		From Lemma \ref{lemma_Greenfunction_7} ($k=0$), we have that
		\begin{equation} \label{new2}
			\begin{split}
				\left(\int_{|\xi|\leq R_{1}}
				|\hat{u}|^2
				{\rm d}\xi\right)^\frac{1}{2}&\le\left(\int_{|\xi|\leq R}
				|\hat{u}|^2
				{\rm d}\xi\right)^\frac{1}{2}\\
				&\le C(1+t)^{-\frac12} + C\int_0^t(1+t-s)^{-\frac12}\|u\|_{L^{2}}(\|\nabla u\|_{L^2} + \|\nabla\tau\|_{L^{2}})
				{\rm d}s.
			\end{split}
		\end{equation}
	
		Then, (\ref{new1}) and (\ref{new2}) yield
		\begin{equation}\label{new_H1_L2_6}
			\begin{split}
				\udt \mathcal{H}_1(t) &+ \frac{\eta_1 R_{1}^{2}}{16}\mathcal{H}_1(t)
				\le C (1+t)^{-1}\\ &+ C \left(\int_0^t(1+t-s)^{-\frac12}\|u\|_{L^{2}}(\|\nabla u\|_{L^2} + \|\nabla\tau\|_{L^{2}})
				{\rm d}s\right)^{2}.
			\end{split}
		\end{equation}
		Substituting \eqref{new_H1_L2_19} into \eqref{new_H1_L2_6}, we obtain
		\begin{equation}\label{new_H1_L2_20}
			\begin{split}
				\udt \mathcal{H}_1(t) &+ \frac{\eta_1 R_{1}^{2}}{16}\mathcal{H}_1(t)
				\le C (1+t)^{-1}\\ &+ C \left(\int_0^t(1+t-s)^{-\frac12}(1+s)^{-\frac12}\|u(s)\|_{L^{2}}(\varphi(s) + \psi(s))
				{\rm d}s\right)^{2}.
			\end{split}
		\end{equation}
		Define that
		\begin{equation}\label{M}
			\mathcal{M}(t): = \sup_{0\leq s \leq t}(1+s)^{\frac{1}{2}}\mathcal{H}_1(s).
		\end{equation}
		Notice that $\mathcal{M}(t)$ is non-decreasing and for all $t\ge 0$,
		$$\mathcal{H}_1(t) \le (1+t)^{-\frac{1}{2}}\mathcal{M}(t)\,\,\,\,\, \text{and}\,\,\,\,\, \|u(t)\|_{L^{2}} \le C(1+t)^{-\frac{1}{4}}\mathcal{M}(t)^{\frac12}.$$	
		Then, \eqref{new_H1_L2_20} and (\ref{M}) immediately yield
		\begin{equation}\label{new_H1_L2_21}
			\begin{split}
				\udt \mathcal{H}_1(t) &+ \frac{\eta_1 R_{1}^{2}}{16}\mathcal{H}_1(t)
				\le C (1+t)^{-1}\\ &+ C \mathcal{M}(t)\left(\int_0^t(1+t-s)^{-\frac12}(1+s)^{-\frac34}(\varphi(s) + \psi(s))
				{\rm d}s\right)^{2}.
			\end{split}
		\end{equation}
	
		Motivated by Dong and Chen (\cite{Dong 2006}),  we define that
		\begin{equation*}
			\mathcal{J}(t):= (1+t)^{\frac14}\int_0^t(1+t-s)^{-\frac12}(1+s)^{-\frac34}(\varphi(s) + \psi(s)){\rm d}s.
		\end{equation*}
		Owing to \eqref{new_H1_L2_19}, for any given small constant $\epsilon$, there exists a $T_\epsilon> 0$, such that
		$$\varphi(T_\epsilon) + \psi(T_\epsilon)\le\epsilon.$$
		Then, we have
		\begin{equation*}\label{new_H1_L2_22}
			\begin{split}
				\mathcal{J}(t) = &\,(1+t)^{\frac14}\int_0^{T_\epsilon}(1+t-s)^{-\frac12}(1+s)^{-\frac34}(\varphi(s) + \psi(s)){\rm d}s \\&+ (1+t)^{\frac14}\int_{T_\epsilon}^t(1+t-s)^{-\frac12}(1+s)^{-\frac34}(\varphi(s) + \psi(s)){\rm d}s\\
				\le&\, (1+t)^{\frac14}\Big(C(T_\epsilon)\int_0^{T_\epsilon}(1+t-s)^{-\frac12}(1+s)^{-\frac34}{\rm d}s + \epsilon\int_{0}^t(1+t-s)^{-\frac12}(1+s)^{-\frac34}{\rm d}s\Big)\\
				\le& \,C(T_\epsilon)(1+t)^{-\frac14} +C\epsilon,
			\end{split}
		\end{equation*}
		for $t\ge 2T_\epsilon.$ Letting $t\rightarrow+\infty$, and using the fact that $\epsilon$ is arbitrarily small, we have
		\begin{equation}\label{new_H1_L2_23}
			\begin{split}
				\mathcal{J}(t)\longrightarrow 0 \,\,\,\,as\,\,\,\,t\rightarrow +\infty.
			\end{split}
		\end{equation}
		Back to \eqref{new_H1_L2_21}, by Gronwall's inequality and using \eqref{new_H1_L2_23}, we have
		\begin{equation*}\label{new_H1_L2_24}
			\begin{split}
				\mathcal{H}_1(t)&\le e^{-Ct}\mathcal{H}_1(0)+ C\int_0^t e^{-C(t-s)}\left((1+s)^{-1} + \mathcal{M}(s)(1+s)^{-\frac12} \mathcal{J}(s)^2 \right) {\rm d}s\\
				&\le C (1+t)^{-1} + C(T_1)\int_0^{T_1} e^{-C(t-s)}(1+s)^{-\frac12} {\rm d}s + C\int_{T_1}^t e^{-C(t-s)}\mathcal{M}(s)(1+s)^{-\frac12} \mathcal{J}(s)^2 {\rm d}s\\
				&\le C(T_1) (1+t)^{-\frac{1}{2}} + \frac12(1+t)^{-\frac12}\sup_{0\leq s \leq t}(1+s)^{\frac{1}{2}}\mathcal{H}_1(s),
			\end{split}
		\end{equation*} for any $t\ge T_2:=2T_1$. Namely, for $t\ge T_2,$ there holds
		\begin{equation}\label{new_H1_L2_26}
			\begin{split}
				(1+t)^{\frac12}\mathcal{H}_1(t)\le C(T_1) + \frac12\sup_{0\leq s \leq t}(1+s)^{\frac{1}{2}}\mathcal{H}_1(s).
			\end{split}
		\end{equation}
		In \eqref{new_H1_L2_26}, taking supremum with respect to $t$ from $T_2$ to $t$, we get
		\begin{equation*}\label{new_H1_L2_27}
			\begin{split}
				\frac12\sup_{T_2\leq s \leq t}(1+s)^{\frac{1}{2}}\mathcal{H}_1(s)\le C,
			\end{split}
		\end{equation*}
		which yields
		\begin{equation}\label{new_H1_L2_28}
			\begin{split}
				\mathcal{H}_1(t)\le C (1+t)^{-\frac12},
			\end{split}
		\end{equation} for $t\ge T_2$. In fact, for all $t\le T_2$, $\mathcal{H}_1(t)$ is bounded. Consequently, for all $t>0$, (\ref{new_H1_L2_28}) also holds. The proof of Lemma \ref{lemma_upper_decay} is complete.
		
	\end{proof}
	
	\medskip
	
	To get sharper decay rate of the quantities on the left-hand side of (\ref{utauH1}), we employ the Fourier splitting method (see \cite{Schonbek 1985}).
	
	\begin{lemma}\label{lemma_upper_decay+1}
		Under the assumptions of Theorem \ref{thm_OB_d_decay}, we have
		\begin{equation}\label{utauH1}
			\| u(t)\|_{H^1}^2 + \| \tau(t)\|_{H^1}^2\le C(1+t)^{-1},
		\end{equation} for all $t>0.$
	\end{lemma}
	\begin{proof}
		Decomposing the term $\|\Lambda u\|_{L^2}^2$ again, there holds
		\begin{equation}\label{new_H1_L2_29}
			\begin{split}
				\|\Lambda u\|_{L^2}^2 = \|\nabla u\|_{L^2}^2 &=  \frac12 \|\nabla u\|_{L^2}^2 + \frac12 \int_{|\xi|\ge g_1(t)}
				|\xi|^2|\hat{u}|^2	{\rm d}\xi + \frac12 \int_{|\xi|\le g_1(t)}
				|\xi|^2|\hat{u}|^2	{\rm d}\xi\\
				&\ge \frac12 \|\nabla u\|_{L^2}^2 + \frac12 g_1(t)^{2}\int_{|\xi|\ge g_1(t)}
				|\hat{u}|^2	{\rm d}\xi,
			\end{split}
		\end{equation}
		where $g_1^2(t)=\frac{24}{\eta_1}(1+t)^{-1}$. Then $g_1^2(t)\le 1$ for all $t\ge \frac{24}{\eta_1}-1$.
		
		Substituting \eqref{new_H1_L2_29} into \eqref{new_H1_L2}, we get
		\begin{equation}\label{new_H1_L2_12}
			\begin{split}
				\udt \mathcal{H}_1(t) + \frac{3}{2}(1+t)^{-1}\mathcal{H}_1(t)
				\le C(1+t)^{-1}\int_{|\xi|\le g_1(t)}
				|\hat{u}|^2	{\rm d}\xi,
			\end{split}
		\end{equation}for all $t\ge \max\{\frac{24}{\eta_1}-1,\frac{6}{\beta}-1\}=:t_2$.
		
		Recalling (\ref{Greenfunction_13-1}), we have
		\begin{equation} \label{duhamel}
			\hat{u}(t) = \,\mathcal{G}_3 \hat{u}_0 + K|\xi|\mathcal{G}_1\hat{\sigma}_0 + \int_0^t\mathcal{G}_3(t-s)\hat{\mathcal{F}}_1 (s) + K|\xi|\mathcal{G}_1(t-s)\hat{\mathcal{F}}_2 (s)
			{\rm d}s,
		\end{equation}
		where
		\begin{eqnarray*}
			\mathcal{F}_1=-\mathbb{P}\left(u\cdot\nabla u\right),\
			\mathcal{F}_2= -\Lambda^{-1}\mathbb{P}{\rm div}\left(u\cdot\nabla\tau\right).
		\end{eqnarray*}
		It is easy to see that
		\begin{equation*}
			g_1^2(t)=\frac{24}{\eta_1}(1+t)^{-1}\le R^2 \doteq \frac{\beta^2}{4\alpha K},
		\end{equation*}
		for all $t\ge	\max\{\frac{96\alpha K}{\eta_1\beta^2}-1, t_2\}=:t_3$.
		
		By virtue of Lemma \ref{lemma_Greenfunction_4}, (\ref{duhamel}) can be estimated as below:
		\begin{align}\label{hatu}
			\nonumber	|\hat{u}|\le \,&Ce^{-\theta|\xi|^2t}|\hat{u}_0| + C |\xi| e^{-\theta|\xi|^2t}|\hat{\sigma}_0|
			+ C\int_0^t	e^{-\theta|\xi|^2(t-s)}|\xi||\widehat{u\otimes u}(s)| {\rm d}s
			\\ \nonumber &+ C\int_0^t  |\xi| e^{-\theta|\xi|^2(t-s)}|\xi||\widehat{u \otimes\tau}(s)|  {\rm d}s\\
			\le \,&C + C|\xi|\int_0^t \|u\|_{L^2}^2{\rm d}s + C|\xi|^2\int_0^t \|u\|_{L^2}\|\tau\|_{L^2}{\rm d}s,
		\end{align}for all $|\xi|\le g_1(t)\le R$ as $t\ge t_3$.
		
		By virtue of \eqref{new_H1_L2_28} and (\ref{hatu}), we get
		\begin{equation}\label{u_L00_low_2}
			\begin{split}
				|\hat{u}(\xi,t)|\le C,
			\end{split}
		\end{equation}  for $t\ge t_3$ and $|\xi|\le g_1(t)$.
		
		Substituting \eqref{u_L00_low_2} into \eqref{new_H1_L2_12}, we get
		\begin{equation}\label{new_H1_L2_15}
			\begin{split}
				\udt \mathcal{H}_1(t) + \frac{3}{2}(1+t)^{-1}\mathcal{H}_1(t)
				\le C(1+t)^{-2},
			\end{split}
		\end{equation}  for all $t\ge t_3$.
		
		Multiplying   \eqref{new_H1_L2_15} by $(1+t)^{\frac32}$, we can deduce that
		\begin{equation*}\label{new_H1_L2_13}
			\begin{split}
				\udt ((1+t)^{\frac32}\mathcal{H}_1(t))\le C(1+t)^{-\frac12},
			\end{split}
		\end{equation*}for all $t\ge t_3$, which yields
		\begin{equation}\label{new_H1_L2_14}
			\begin{split}
				\mathcal{H}_1(t)\le C (1+t)^{-1}.
			\end{split}
		\end{equation}
		Similarly, since $\mathcal{H}_1(t)$ is bounded for all $t\le t_3$, thus (\ref{new_H1_L2_14}) also holds for all $t>0$.
	\end{proof}
	
	\medskip	
	
	Next, we will develop a way to capture the optimal time-decay rates for
	the higher-order derivatives of the solution.

	\begin{lemma}\label{lemma_upper_decay_2}
		Under the assumptions of Theorem \ref{thm_OB_d_decay}, we have
		\begin{equation}\label{upper decay2}
			\|\nabla u(t)\|_{H^2}^2 + \|\nabla \tau(t)\|_{H^2}^2\le C (1+t)^{-2},
		\end{equation}
		for all $t>0$.
	\end{lemma}
	\begin{proof}
		Summing \eqref{est_first}, \eqref{est_second} and $\eta_2$\eqref{est_second_u} ($\mu=0$) up, we obtain that
		\begin{equation}\label{new_H2_L2_1}
			\begin{split}
				\udt \mathcal{H}_2(t) +  \frac{\beta K}{2}\|\nabla\tau\|_{H^1}^2 +\frac{\eta_2\alpha}{8}\|\Lambda^2 u\|_{L^2}^2	\le C\eta_2\big(\|\nabla u\|_{L^\infty}^{2}+\|\nabla \tau\|_{L^\infty}^{2}\big)\|\nabla u\|_{L^2}^2.
			\end{split}
		\end{equation}Then, \eqref{new_H2_L2_1} yields
		\begin{equation}\label{new_H2_L2_2}
			\begin{split}
				\udt \mathcal{H}_2(t) +  \bar{c}_0\mathcal{H}_2(t) \le C\|\nabla u\|_{L^2}^2,
			\end{split}
		\end{equation} for some positive constants $\bar{c}_0$ and $C$.
		
		By virtue of \eqref{utauH1}, \eqref{new_H2_L2_2} yields
		\begin{equation}\label{new_H2_L2_3}
			\begin{split}
				\mathcal{H}_2(t)\le &C(1+t)^{-1},
			\end{split}
		\end{equation}for all $t>0$.
		
		Similarly, summing \eqref{est_second}, \eqref{est_third} and $\eta_3$\eqref{est_third_u} ($\mu=0$) up, we obtain that
		\begin{equation} \label{new5}
			\udt \mathcal{H}_3(t) +  \frac{\beta K}{2}\|\nabla^2\tau\|_{H^1}^2 +\frac{\eta_3\alpha}{8}\|\Lambda^3 u\|_{L^2}^2	\le C\big(\|\nabla u\|_{L^\infty}+\|\nabla \tau\|_{L^\infty}\big)\|\nabla^2 u\|_{L^2}^2,
		\end{equation}which yields
		\begin{equation}\label{new_H3_L2_2}
			\begin{split}
				\udt \mathcal{H}_3(t) +  \bar{c}_1\mathcal{H}_3(t) \le C\|\nabla^2 u\|_{L^2}^2,
			\end{split}
		\end{equation}for some positive constants $\bar{c}_1$ and $C$.
		
		Combining \eqref{new_H2_L2_3} with \eqref{new_H3_L2_2}, we get
		\begin{equation} \label{new3}
			\mathcal{H}_3(t)\le C(1+t)^{-1},
		\end{equation} for all $t>0$.
		
		Similar to \eqref{new_H1_L2_29}, we have
		\begin{equation}\label{na2u}
			\|\Lambda^2 u\|_{L^2}^2 = \|\nabla^2 u\|_{L^2}^2 \ge \frac12 \|\nabla^2 u\|_{L^2}^2 + \frac12 g_2^2(t)\int_{|\xi|\ge g_2(t)}
			|\xi|^2|\hat{u}|^2	{\rm d}\xi,
		\end{equation}
		where $g_2(t)>0$ is to be determined.
		
		Substituting (\ref{na2u}) into \eqref{new_H2_L2_1}, and using (\ref{new_H2_L2_3}) and (\ref{new3}), we have
		\begin{equation}\label{new_H2_L2_4}
			\begin{split}
				&\udt \mathcal{H}_2(t) +  \frac{\beta K}{2}\|\nabla\tau\|_{H^1}^2 +\frac{\eta_2\alpha}{16}\|\nabla^2 u\|_{L^2}^2 + \frac{\eta_2\alpha}{16}g_2^2(t)\|\nabla u\|_{L^2}^2\\
				\le\,&\frac{\eta_2\alpha}{16}g_2^2(t)\int_{|\xi|\le g_2(t)}
				|\xi|^2|\hat{u}|^2	{\rm d}\xi + C\eta_2\big(\|\nabla u\|_{L^\infty}^{2}+\|\nabla \tau\|_{L^\infty}^{2}\big)\|\nabla u\|_{L^2}^2\\
				\le\,&\frac{\eta_2\alpha}{16}g_2^2(t)\int_{|\xi|\le g_2(t)}
				|\xi|^2|\hat{u}|^2	{\rm d}\xi +\frac{C\eta_2}{1+t}\|\nabla u\|_{L^2}^2.
			\end{split}
		\end{equation}
		Here, taking $g_2^2(t)=\frac{160}{\eta_2}(1+t)^{-1}$, then $g_2^2(t)\le 1$ for all $t\ge \frac{160}{\eta_2}-1=:t_4$. In addition, letting $\eta_2 \le \min \{16\beta, \frac{5\alpha}{C}\}$, then \eqref{new_H2_L2_4}  yields
		\begin{equation}\label{new4}
			\begin{split}
				\udt \mathcal{H}_2(t) + \frac{\eta_2}{64}g_2^2(t)\,(2\alpha\|\nabla u\|_{H^1}^2 + 2K\| \nabla\tau\|_{H^1}^2)
				\le C(1+t)^{-3},
			\end{split}
		\end{equation}where we have used the boundedness of $|\hat{u}|$ for all $|\xi|\le g_2(t)$, which is similar to (\ref{u_L00_low_2}).
		
		Combining (\ref{new4}) with \eqref{bu_9}, we obtain
		\begin{equation}\label{new_H2_L2_6}
			\begin{split}
				\udt \mathcal{H}_2(t) + \frac{\eta_2}{64}g_2^2(t)\mathcal{H}_2(t)
				\le C(1+t)^{-3}.
			\end{split}
		\end{equation}
		Multiplying   \eqref{new_H2_L2_6} by $(1+t)^\frac52$, we can deduce that for all $t\ge t_4$,
		\begin{equation}\label{new_H2_L2_7}
			\begin{split}
				\mathcal{H}_2(t)\le &C(1+t)^{-2}.
			\end{split}
		\end{equation}
		Substituting \eqref{new_H2_L2_7} into  \eqref{new_H3_L2_2}, we can also deduce that for all $t\ge t_4$,
		\begin{equation}\label{new_H3_L2_8}
			\begin{split}
				\mathcal{H}_3(t)\le &C(1+t)^{-2}.
			\end{split}
		\end{equation}
		Since for all $t\le t_4$, $\mathcal{H}_2(t)$ and $\mathcal{H}_3(t)$ are bounded, thus \eqref{new_H2_L2_7} and \eqref{new_H3_L2_8} also hold for all $t>0$.
	\end{proof}

	\begin{corollary}\label{cor1}
		Under the assumptions of Theorem \ref{thm_OB_d_decay}, we have
		\begin{equation}\label{new_tau_L2_2}
			\begin{split}
				\|\tau(t)\|_{L^2}^2\le C(1+t)^{-2},
			\end{split}
		\end{equation}for all $t>0.$
	\end{corollary}
	\begin{proof}
		Applying $\nabla^k$ ($k=0,1,2$) to (\ref{Oldroyd_B_d})$_2$, multiplying the result by $\nabla^k \tau$, and integrating with respect to $x$, we have that
		\begin{equation}\label{new_tau_L2_1}
			\begin{split}
				&\frac{1}{2}\frac{\mathrm{d}}{\mathrm{d}t} \|\nabla^k \tau\|_{L^2}^2  + \frac{\beta}{2}\|\nabla^k \tau\|_{L^2}^2	\\&\le C   \|\nabla^{k+1}u\|_{L^2}^2 + C\|\nabla^{k}(u\cdot \nabla\tau)\|_{L^2}^2\\
				&\le  C\|\nabla^{k+1}u\|_{L^2}^2 + C\|\nabla^{k+1}\tau\|_{L^2}^2\|u\|_{L^\infty}^2 + \|\nabla \tau\|_{L^\infty}^2\|\nabla^{k}u\|_{L^2}^2.
			\end{split}
		\end{equation}
	However, (\ref{new_tau_L2_1}) for $k=1,2$ will be used later. In fact, to prove \eqref{new_tau_L2_2}, it suffices to take $k=0$ in \eqref{new_tau_L2_1}. Then by virtue of (\ref{new_H2_L2_7}) and \eqref{new_H3_L2_8}, the inequality \eqref{new_tau_L2_2} holds.
		
	\end{proof}
	\begin{lemma}\label{lemma_upper_decay_3}
		Under the assumptions of Theorem \ref{thm_OB_d_decay}, we have
		\begin{equation}\label{na2udecay}
			\|\nabla^2 u(t)\|_{H^1}^2 + \|\nabla^2 \tau(t)\|_{H^1}^2\le C(1+t)^{-3},
		\end{equation} for all $t>0.$
	\end{lemma}
	\begin{proof} By virtue of (\ref{upper decay2}) and (\ref{new5}), we have
		\begin{equation*}\label{new_H3_L2_2+1}
			\begin{split}
				\udt \mathcal{H}_3(t) +  \bar{c}_2\mathcal{H}_3(t) \le& C\big(\|\nabla u\|_{L^\infty}+\|\nabla \tau\|_{L^\infty}\big)\|\nabla^2 u\|_{L^2}^2+C\int_{|\xi|\leq R}|\xi|^{4}|\hat{u}(t)|^2{\rm d}\xi\\ \le& C\big(\|\nabla u\|_{H^2}^3+\|\nabla \tau\|_{H^2}^3\big) +C\int_{|\xi|\leq R}|\xi|^{4}|\hat{u}(t)|^2{\rm d}\xi\\ \le& C(1+t)^{-3} +C\int_{|\xi|\leq R}|\xi|^{4}|\hat{u}(t)|^2{\rm d}\xi,
			\end{split}
		\end{equation*}for some positive constants $\bar{c}_2$ and $C$, which together with Lemma \ref{lemma_Greenfunction_7} yields
		\begin{equation}\label{new_H3_L2_2+2}
			\begin{split}
				\udt \mathcal{H}_3(t) +  \bar{c}_2\mathcal{H}_3(t) \le& C(1+t)^{-3}.
			\end{split}
		\end{equation}
		Then (\ref{na2udecay}) can be obtained by using (\ref{new_H3_L2_2+2}). The proof of Lemma \ref{lemma_upper_decay_3} is complete.
		
	\end{proof}

	\begin{lemma}\label{lemma_upper_decay_4}
		Under the assumptions of Theorem \ref{thm_OB_d_decay}, we have
		\begin{equation}\label{na3udecay}
			\begin{split}
				\|\nabla^3 u(t)\|_{L^2}^2 + \|\nabla^3 \tau(t)\|_{L^2}^2\le C (1+t)^{-4},
			\end{split}
		\end{equation}for all $t>0.$
	\end{lemma}
	\begin{proof}
		Here, we choose the standard cut-off function  $0\le \varphi_0(\xi)\le 1$ in $C_c^\infty(\mathbb{R}^2)$ such that
		\begin{equation*}
			\varphi_0(\xi) =
			\begin{cases}
				1, \,&for\, |\xi|\le \frac{R}{2},\\
				0, \,&for\, |\xi|\ge R,
			\end{cases}
		\end{equation*}
		where $R$ is  defined in Lemma \ref{lemma_Greenfunction_4}. The low-high-frequency decomposition $(f^l(x), f^h(x))$ for a function $f(x)$ is stated as follows:
		\begin{equation*}
			f^l(x):=\mathcal{F}^{-1}(\varphi_0(\xi)\hat{f}(\xi))\,\,\,\ \text{and} \,\,\,\,f^h(x):=f(x)-f^l(x).
		\end{equation*}	
		Multiplying $\Lambda^3$(\ref{u_sigma_1})$_1$ and $\Lambda^2(\ref{u_sigma_1})^l_2$  by $-\Lambda^2\sigma^l$ and  $-\Lambda^3 u$, respectively, summing the results up, and using integration by parts, we have
		\begin{align}\label{bu_15}
			\begin{split}
				&-\partial_t\langle\Lambda^3 u,\Lambda^2 \sigma^l\rangle\\
				= & \,-\Big(K\langle\Lambda^3\sigma, \Lambda^3 \sigma^l\rangle - \langle\beta\Lambda^2\sigma^l,\Lambda^3 u \rangle - \frac{\alpha}{2}\langle\Lambda^3 u, \Lambda^3 u^l\rangle \Big)\\ &+ \Big(\langle \Lambda^3\mathbb{P}(u\cdot \nabla u),\Lambda^2\sigma^l\rangle + \langle\big(\Lambda \mathbb{P}\udiv(u\cdot \nabla \tau)\big)^l,\Lambda^3 u\rangle\Big)\\
				=: &\,I_7 + I_8.
			\end{split}
		\end{align}	
		For $I_7$ and $I_8$, using (\ref{bu_7}), we have
		\begin{align}
			|I_7|\le\, &K\|\Lambda^3 \sigma\|_{L^2}^2 + \frac{\alpha}{32}\|\Lambda^3 u\|_{L^2}^2 + \frac{8\beta^2}{\alpha}\|\Lambda^2\sigma\|_{L^2}^2 + \frac{\alpha}{2}\|\Lambda^3 u^l\|_{L^2}^2 + \frac{\alpha}{8}\|\Lambda^3 u\|_{L^2}^2,\label{I7}\\
			\nonumber|I_8|\le\, &\frac12\|\Lambda^3 \sigma\|_{L^2}^2
			+  \frac12\|\Lambda^2\mathbb{P}(u\cdot \nabla u)\|_{L^2}^2
			+  \frac{\alpha}{32}\|\Lambda^3 u\|_{L^2}^2 + \frac{8}{\alpha}\|\Lambda\mathbb{P}\udiv(u\cdot \nabla \tau)\|_{L^2}^2\\
			\le\, &\frac12\|\Lambda^3 \sigma\|_{L^2}^2
			+  C\|\nabla^2(u\cdot \nabla u)\|_{L^2}^2 + \frac{\alpha}{32}\|\Lambda^3 u\|_{L^2}^2 + C\|\nabla^2(u\cdot \nabla \tau)\|_{L^2}^2.\label{I8}
		\end{align}
		Substituting (\ref{I7}) and (\ref{I8})  into (\ref{bu_15}), we  get
		\begin{equation}\label{new_H5_L2_1}
			\begin{split}
				\begin{split}
					&-\partial_t\langle\Lambda^3 u,\Lambda^2 \sigma^l\rangle\\
					\leq \,& \frac{\alpha}{2}\|\Lambda^3 u^l\|_{L^2}^2 + \big(\frac{3\alpha}{16} + C\| u\|_{L^\infty}^{2}\big)\|\Lambda^3 u\|_{L^2}^2 + (K + \frac{8\beta^2}{\alpha} + \frac12)\|\nabla^2\sigma\|_{H^1}^2\\
					\, &+    C\big(\|u\|_{L^\infty}^{2}+\|\nabla u\|_{L^\infty}^{2}+\|\nabla \tau\|_{L^\infty}^{2}\big)\big(\|\nabla^2\tau\|_{H^1}^2 + \|\nabla^2 u\|_{L^2}^2\big),
				\end{split}
			\end{split}
		\end{equation}
		where \eqref{bu_21} is used.
		
		Summing \eqref{est_third_u} ($\mu=0$) and \eqref{new_H5_L2_1} up, and using the smallness of the solution stated in Theorem \ref{wellposedness}, we get
		\begin{equation}\label{new_H5_L2_2}
			\begin{split}
				&\partial_t\langle\Lambda^3 u,\Lambda^2 \sigma^h\rangle + \frac{\alpha}{16}\|\Lambda^3 u\|_{L^2}^2\\
				\leq \, &\frac{\alpha}{2}\|\Lambda^3 u^l\|_{L^2}^2  + (2K + \frac{12\beta^2}{\alpha} + 1)\|\nabla^2\sigma\|_{H^1}^2 \\
				&  + C\big(\|u\|_{L^\infty}^{2}+\|\nabla u\|_{L^\infty}^{2}+\|\nabla \tau\|_{L^\infty}^{2}\big)\big(\|\nabla^2\tau\|_{H^1}^2 + \|\nabla^2 u\|_{L^2}^2\big).
			\end{split}
		\end{equation}
		Letting $\eta_4>0$ small enough, then summing 2$\times$\eqref{est_third} ($\mu=0$) and $\eta_4$\eqref{new_H5_L2_2} up, we have
		\begin{equation}\label{new_H5_L2_4}
			\begin{split}
				&\udt \mathcal{H}_4(t) + \frac{\eta_4\alpha}{16}\|\Lambda^3 u\|_{L^2}^2 + \frac{\beta K}{100}\|\nabla^3\tau\|_{L^2}^2\\
				\leq\, & C\big(\|\nabla u\|_{L^\infty}+\|\nabla \tau\|_{L^\infty}\big)\|\nabla^3 u\|_{L^2}^2 + \eta_4\Big(\frac{\alpha}{2} \|\Lambda^3 u^l\|_{L^2}^2
				+ (2K + \frac{12\beta^2}{\alpha} + 1)\|\nabla^2\sigma\|_{H^1}^2 \\
				\, &+ C\big(\|u\|_{L^\infty}^{2}+\|\nabla u\|_{L^\infty}^{2}+\|\nabla \tau\|_{L^\infty}^{2}\big)\big(\|\nabla^2\tau\|_{H^1}^2 + \|\nabla^2 u\|_{L^2}^2\big)\Big),
			\end{split}
		\end{equation} where
		\begin{equation*}
			\mathcal{H}_4:=\alpha\|\nabla^3 u\|_{L^2}^2 + K\|\nabla^3\tau\|_{L^2}^2+\eta_4\langle\Lambda^3 u,\Lambda^2 \sigma^h\rangle=O(\|\nabla^3 u\|_{L^2}^2+\|\nabla^3 \tau\|_{L^2}^2).
		\end{equation*}
		By virtue of (\ref{bu_7}), (\ref{bu_8}) and (\ref{upper decay2}) and the smallness of the solution and $\eta_4$, (\ref{new_H5_L2_4}) yields
		\begin{equation}\label{new_H5_L2_5}
			\begin{split}
				\begin{split}
					&\udt \mathcal{H}_4(t) + \frac{\eta_4\alpha}{32}\|\Lambda^3 u\|_{L^2}^2 + \frac{\beta K}{100}\|\nabla^3\tau\|_{L^2}^2\\
					\leq \,& \eta_4\big(\frac{\alpha}{2} \|\Lambda^3 u^l\|_{L^2}^2
					+ (2K + \frac{12\beta^2}{\alpha} + 1)\|\nabla^2\sigma\|_{L^2}^2\\
					\,&+  C\big(\|u\|_{L^\infty}^{2}+\|\nabla u\|_{L^\infty}^{2}+\|\nabla \tau\|_{L^\infty}^{2}\big)\big(\|\nabla^2\tau\|_{L^2}^2 + \|\nabla^2 u\|_{L^2}^2\big).
				\end{split}
			\end{split}
		\end{equation}
		Using \eqref{bu_8}, we have that
		\begin{equation}\label{bu_17}
			\frac{\beta K}{200}\|\nabla^3\tau\|_{L^2}^2\ge \frac{1}{C}\|\nabla^3\sigma\|_{L^2}^2  \ge \frac{1}{C} R^2\int_{|\xi|\ge R}
			|\xi|^4|\hat{\sigma}|^2	{\rm d}\xi\ge \frac{1}{C}\|\nabla^2\sigma^h\|_{L^2}^2.
		\end{equation}
		In addition, letting
		\begin{equation}\label{bu_18}
			\eta_4 \le \frac{C}{2K + \frac{12\beta^2}{\alpha} + 1},
		\end{equation}
		substituting \eqref{bu_17} and \eqref{bu_18} into  \eqref{new_H5_L2_5}, and using (\ref{utauH1}), \eqref{upper decay2} and \eqref{na2udecay}, we get
		\begin{equation}\label{new_H5_L2_6}
			\begin{split}
				\begin{split}
					\udt \mathcal{H}_4(t) + \bar{c}_3\mathcal{H}_4(t)
					\leq& \, C\|\Lambda^3 u^l\|_{L^2}^2
					+ C\|\nabla^2\sigma^l\|_{L^2}^2\\
					&+  C\big(\|u\|_{L^\infty}^{2}+\|\nabla u\|_{L^\infty}^{2}+\|\nabla \tau\|_{L^\infty}^{2}\big)\big(\|\nabla^2\tau\|_{L^2}^2 + \|\nabla^2 u\|_{L^2}^2\big)\\
					\leq& \, C\|\Lambda^3 u^l\|_{L^2}^2
					+ C\|\nabla^2\sigma^l\|_{L^2}^2 + C(1+t)^{-4},
				\end{split}
			\end{split}
		\end{equation}
		for some positive constants $\bar{c}_3$ and $C$.
		
		By virtue of (\ref{Greenfunction_13-1}), (\ref{utauH1}), and (\ref{upper decay2}), we have
		\begin{align}\label{new_H5_L2_8}
			&\left(\int_{|\xi|\leq R}
			|\xi|^{6}|\hat{u}(t)|^2{\rm d}\xi\right)^\frac{1}{2}\notag\\
			\le & ~C(1+t)^{-2} + C\int_0^t\Big(\int_{|\xi|\leq R}|\xi|^{6}e^{-2\theta|\xi|^2(t-s)}(|\hat{\mathcal{F}}_1 (\xi,s)|^2 + |\hat{\mathcal{F}}_2 (\xi,s)|^2){\rm d}\xi\Big)^\frac12{\rm d}s\notag\\
			\le &~C(1+t)^{-2} +C\int_0^{\frac{t}{2}}(1+t-s)^{-2}(\|\hat{\mathcal{F}}_1 (\cdot,s)\|_{L^\infty} + \|\hat{\mathcal{F}}_2 (\cdot,s)\|_{L^\infty}){\rm d}s \\
			&+C\int_{\frac{t}{2}}^{t}\Big(\int_{|\xi|\leq R}|\xi|^{4}\frac{(1+t-s)^{2}}{(1+t-s)^{2}}e^{-2\theta|\xi|^2(1+t-s)}|\xi|^2(|\hat{\mathcal{F}}_1 (\xi,s)|^2 + |\hat{\mathcal{F}}_2 (\xi,s)|^2){\rm d}\xi\Big)^\frac12{\rm d}s\notag\\
			\le &~C(1+t)^{-2} +C\int_{\frac{t}{2}}^{t}(1+t-s)^{-1}\Big(\big\||\xi|\hat{\mathcal{F}}_1 (\xi,s)\big\|_{L^2} + \big\||\xi|\hat{\mathcal{F}}_2 (\xi,s)\big\|_{L^2}  \Big){\rm d}s,\notag
		\end{align} where we have used
		\begin{equation*}
			\begin{split}
				\|\hat{\mathcal{F}}_1(\xi,s) \|_{L^\infty} + \|\hat{\mathcal{F}}_2(\xi,s)\|_{L^\infty}\le& C\|u(s)\|_{L^2}\big(\|\nabla u(s)\|_{L^2}+\|\nabla \tau(s)\|_{L^2}\big)
				\\ \le& C(1+s)^{-\frac{3}{2}}.
			\end{split}
		\end{equation*}		
		Using \eqref{upper decay2} and \eqref{na2udecay} and Gagliardo-Nirenberg inequailty, we have
		\begin{equation}\label{new_H5_L2_9}
			\begin{split}
				\big\||\xi|\hat{\mathcal{F}}_1 (\xi,s)\big\|_{L^2}&\le C\|\nabla(u\cdot\nabla u)\|_{L^2}\\
				&\le C\big(\| u\|_{L^\infty}\|\nabla^2 u\|_{L^2} + \|\nabla u\|_{L^\infty}\|\nabla u\|_{L^2}\big)\\
				&\le C\big(\| u\|_{L^2}^{\frac12}\|\nabla^2 u\|_{L^2}^{\frac12}\|\nabla^2 u\|_{L^2} + \|\nabla u\|_{L^2}^{\frac12}\|\nabla^3 u\|_{L^2}^{\frac12}\|\nabla u\|_{L^2}\big)\\
				&\le C(1+s)^{-\frac{9}{4}}.
			\end{split}
		\end{equation}
		Similarly, we have
		\begin{equation}\label{new_H5_L2_10}
			\begin{split}
				\||\xi|\hat{\mathcal{F}}_2 (\xi,s)\|_{L^2}&\le C(1+s)^{-\frac{9}{4}}.
			\end{split}
		\end{equation}
		Subtituting \eqref{new_H5_L2_9} and \eqref{new_H5_L2_10} into  \eqref{new_H5_L2_8}, we can deduce that
		\begin{equation}\label{new_H5_L2_11}
			\begin{split}
				\|\Lambda^3 u^l\|_{L^2}^2&\le C(1+s)^{-4}.
			\end{split}
		\end{equation}
		It is worth noticing that $\|\nabla^2\sigma^l\|_{L^2}$ has the similar structure as $\|\Lambda^3 u^l\|_{L^2}$ (see the proof of Lemma \ref{lemma_Greenfunction_7}). Thus we have
		\begin{equation}\label{new_H5_L2_12}
			\begin{split}
				\|\nabla^2\sigma^l\|_{L^2}^2&\le C(1+s)^{-4}.
			\end{split}
		\end{equation}
		Substituting \eqref{new_H5_L2_11} and \eqref{new_H5_L2_12} into  \eqref{new_H5_L2_6}, we get
		\begin{equation}\label{new_H5_L2_13}
			\begin{split}
				\begin{split}
					\udt \mathcal{H}_4(t) + \bar{c}_3\mathcal{H}_4(t)\le C(1+t)^{-4}.
				\end{split}
			\end{split}
		\end{equation}
		Then, \eqref{na3udecay} can be obtained by using \eqref{new_H5_L2_13}.
	\end{proof}
	\begin{corollary}\label{cor2}
		Under the assumptions of Theorem \ref{thm_OB_d_decay}, we have
		\begin{equation}\label{new_tau_L2_}
			\begin{split}
				\|\nabla\tau(t)\|_{L^2}^2\le C (1+t)^{-3},\,\,\,\,
				\|\nabla^2\tau(t)\|_{L^2}^2\le C (1+t)^{-4}.
			\end{split}
		\end{equation}for all $t>0.$
	\end{corollary}
	\begin{proof}
		\eqref{new_tau_L2_} can be obtained by using (\ref{upper decay2}), \eqref{new_tau_L2_1}, \eqref{na2udecay}, and \eqref{na3udecay}.
	\end{proof}
	
	\medskip
	Noticing that \eqref{new_tau_L2_1} can not be used to get further time-decay rate of the quantity $\|\nabla^3 \tau(t)\|_{L^2}$, since the decay of $\|\nabla^4 u(t)\|_{L^2}$ is unknown. Here we consider a combination of $\|\nabla^3 u^h(t)\|_{L^2}^2$ and $\|\nabla^3 \tau(t)\|_{L^2}^2$.
	\begin{lemma}\label{lemma_upper_decay_5}
		Under the assumptions of Theorem \ref{thm_OB_d_decay}, we have
		\begin{equation}\label{highest_order_decay_0}
			\begin{split}
				\|\nabla^3 u^h(t)\|_{L^2}^2 + \|\nabla^3 \tau(t)\|_{L^2}^2\le C (1+t)^{-5},
			\end{split}
		\end{equation}
		for all $t>0.$
	\end{lemma}
	\begin{proof}
		Multiplying $\nabla^3$(\ref{Oldroyd_B_d})$_1^h$ and $\nabla^3$ (\ref{Oldroyd_B_d})$_2$ by $\alpha  \nabla^3 u^h$ and  $K  \nabla^3 \tau$, respectively, summing the results up, and using integration by parts, we have
		\begin{equation}\label{highest_order_decay_1}
			\begin{split}
				&\frac12 \udt (\alpha\|\nabla^3 u^h\|_{L^2}^2 + K\|\nabla^3\tau\|_{L^2}^2) + \beta K\|\nabla^3\tau\|_{L^2}^2\\
				= & -\, \Big(\langle\alpha \nabla^3(u\cdot\nabla u)^h,\nabla^3u^h \rangle + \langle K \nabla^3(u\cdot\nabla\tau),\nabla^3\tau \rangle\Big)\\
				& + \, \Big(\langle \alpha \nabla^3\mathbb{D}(u^h),K\nabla^3\tau^l \rangle + \langle \alpha \nabla^3\mathbb{D}(u^l),K\nabla^3\tau \rangle\Big)\\
				=: &\,I_{9} + I_{10},
			\end{split}
		\end{equation}
		where we have used $$\langle K \nabla^3{\rm div}\tau^h,\alpha\nabla^3u^h \rangle = -\,\langle \alpha \nabla^3\mathbb{D}(u^h),K\nabla^3\tau^h \rangle.$$
		Using the Plancherel's Theorem, we obtain
		\begin{equation}\label{highest_order_decay_2}
			\begin{split}
				\langle\alpha \nabla^3(u\cdot\nabla u)^h,\nabla^3u^h \rangle
				= &\,\langle\alpha|\xi|^3 \big(1-\varphi_0(\xi)\big)\widehat{u\cdot\nabla u},\,\,\,|\xi|^3\big(1-\varphi_0(\xi)\big)\hat{u} \rangle\\
				= &\,\langle\alpha|\xi|^3\widehat{u\cdot\nabla u},\,\,\,|\xi|^3(1-\varphi_0(\xi))^2\hat{u}\rangle\\
				= &\,\langle\alpha\nabla^3(u\cdot\nabla u),\,\,\nabla^3u^{\widetilde{h}}\rangle,
			\end{split}
		\end{equation}
		where we define that
		\begin{equation*}
			u^{\widetilde{h}}(x):=\mathcal{F}^{-1}((1-\varphi_0(\xi))^2\hat{u}(\xi))\,\,\,\ \text{and} \,\,\,\,u^{\widetilde{l}}(x):=u(x)-u^{\widetilde{h}}(x).
		\end{equation*}
	
		Next, we can divide the term $\langle\alpha\nabla^3(u\cdot\nabla u),\nabla^3u^{\widetilde{h}}\rangle$ in (\ref{highest_order_decay_2}) into two parts, and using  (\ref{upper decay2}),  (\ref{na2udecay}) and  (\ref{na3udecay}), we have
		\begin{equation}\label{highest_order_decay_3}
			\begin{split}
				&\big|\langle\alpha\nabla^3(u\cdot\nabla u),\nabla^3u^{\widetilde{h}}\rangle\big|\\
				= &\,\big|\langle\alpha\nabla^3(u\cdot\nabla u^{\widetilde{h}}),\nabla^3u^{\widetilde{h}}\rangle + \langle\alpha\nabla^3(u\cdot\nabla u^{\widetilde{l}}),\nabla^3u^{\widetilde{h}}\rangle\big|\\
				\le &\,C\Big(\|\nabla u\|_{L^\infty}\|\nabla^3 u^{\widetilde{h}}\|_{L^2}^2 + \|\nabla u^{\widetilde{h}}\|_{L^\infty}\|\nabla^3 u\|_{L^2}\|\nabla^3 u^{\widetilde{h}}\|_{L^2}\\&+\|\nabla^2u\|_{H^1}\|\nabla^2u^{\widetilde{h}}\|_{H^1}\|\nabla^3u^{\widetilde{h}}\|_{L^2}
				+ \|\nabla u\|_{L^\infty}\|\nabla^3 u^{\widetilde{l}}\|_{L^2}\|\nabla^3 u^{\widetilde{h}}\|_{L^2}\\ & + \|\nabla u^{\widetilde{l}}\|_{L^\infty}\|\nabla^3 u\|_{L^2}\|\nabla^3 u^{\widetilde{h}}\|_{L^2}+ \|\nabla^2 u^{\widetilde{l}}\|_{H^1}\|\nabla^2 u\|_{H^1}\|\nabla^3 u^{\widetilde{h}}\|_{L^2} \\&+\|u\|_{L^\infty}\|\nabla^4 u^{\widetilde{l}}\|_{L^2}\|\nabla^3 u^{\widetilde{h}}\|_{L^2}\Big)\\
				\le& C(1+t)^{-5} + C(1+t)^{-\frac{5}{2}}\|\nabla^4 u^{\widetilde{l}}\|_{L^2}.
			\end{split}
		\end{equation}
		Similar with (\ref{new_H5_L2_8}) and (\ref{new_H5_L2_9}), we have
		\begin{equation}\label{highest_order_decay_}
			\begin{split}
				\|\nabla^4 u^{\widetilde{l}}\|_{L^2}&\le C(1+t)^{-\frac{5}{2}} + C\int_{\frac{t}{2}}^{t}(1+t-s)^{-1}\Big(\||\xi|^2\hat{\mathcal{F}}_1 (\xi,s)\|_{L^2} + \||\xi|^2\hat{\mathcal{F}}_2 (\xi,s)\|_{L^2}  \Big){\rm d}s\\
				&\le C(1+t)^{-\frac{5}{2}} + C\int_{\frac{t}{2}}^{t}(1+t-s)^{-1}(\|\nabla^2(u\cdot\nabla u)\|_{L^2} + \|\nabla^2(u\cdot\nabla \tau)\|_{L^2}) {\rm d}s\\
				&\le C(1+s)^{-\frac{5}{2}}.
			\end{split}
		\end{equation}		
		Substituting \eqref{highest_order_decay_3} and \eqref{highest_order_decay_} into  \eqref{highest_order_decay_1}, we have
		\begin{equation}\label{highest_order_decay_4}
			\begin{split}
				|I_{9}|&\le C\,(1+t)^{-5} + C (\|\nabla u\|_{L^\infty}\|\nabla^3\tau\|_{L^2}^2 + \|\nabla \tau\|_{L^\infty}\|\nabla^3u\|_{L^2}\|\nabla^3\tau\|_{L^2})\\
				&\le C\,(1+t)^{-5}.
			\end{split}
		\end{equation}
		Noticing that
		\begin{equation*}\label{highest_order_decay_5}
			\begin{split}
				\langle \alpha \nabla^3\mathbb{D}(u^h),K\nabla^3\tau^l \rangle = \langle \alpha \nabla^3\mathbb{D}(u^l),K\nabla^3\tau^h \rangle
				\le\frac{\beta K}{4}\|\nabla^3 \tau\|_{L^2}^2 + \frac{\alpha^2 K}{\beta}\|\nabla^4 u^{l}\|_{L^2}^2,
			\end{split}
		\end{equation*}
		and
		$$\langle \alpha \nabla^3\mathbb{D}(u^l),K\nabla^3\tau \rangle\le\frac{\beta K}{4}\|\nabla^3 \tau\|_{L^2}^2 + \frac{\alpha^2 K}{\beta}\|\nabla^4 u^{l}\|_{L^2}^2,$$ we can estimate $I_{10}$ as follows:
		\begin{equation}\label{highest_order_decay_6}
			\begin{split}
				|I_{10}|\le \frac{\beta K}{2}\|\nabla^3 \tau\|_{L^2}^2 + \frac{2\alpha^2 K}{\beta}\|\nabla^4 u^{l}\|_{L^2}^2.
			\end{split}
		\end{equation}
		Substituting  \eqref{highest_order_decay_4}  and \eqref{highest_order_decay_6} into \eqref{highest_order_decay_1},  we get
		\begin{equation}\label{highest_order_decay_7}
			\begin{split}
				\frac12 \udt (\alpha\|\nabla^3 u^h\|_{L^2}^2 + K\|\nabla^3\tau\|_{L^2}^2) + \frac{\beta K}{2}\|\nabla^3\tau\|_{L^2}^2\le C\,(1+t)^{-5},
			\end{split}
		\end{equation} where the estimate of the second term on the right-hand side of (\ref{highest_order_decay_6}) is similar to that of (\ref{highest_order_decay_}).
		
		Multiplying $\Lambda^3$(\ref{u_sigma_1})$_1^h$ and $\Lambda^2(\ref{u_sigma_1})_2^h$ by $\Lambda^2\sigma^h$ and  $\Lambda^3 u^h$, respectively, summing the results up, and using integration by parts, we have
		\begin{equation}\label{highest_order_decay_8}
			\begin{split}
				&\partial_t\langle\Lambda^3 u^h,\Lambda^2 \sigma^h\rangle + \frac{\alpha}{2}\|\Lambda^3 u^h\|_{L^2}^2\\
				= & \,\Big(K\|\Lambda^3 \sigma^h\|_{L^2}^2  - \langle\beta\Lambda^2\sigma^h,\Lambda^3 u^h \rangle\Big)\\ &- \Big(\langle \Lambda^3\mathbb{P}(u\cdot \nabla u)^h,\Lambda^2\sigma^h\rangle + \langle \Lambda\mathbb{P}\udiv(u\cdot \nabla \tau)^h,\Lambda^3 u^h\rangle\Big)\\
				=: &\,I_{11} - I_{12}.
			\end{split}
		\end{equation}
		Then, $I_{11}$ and $I_{12}$ can be estimated as follows:
		\begin{align}
			\label{new6}	|I_{11}|&\le K\|\Lambda^3 \sigma^h\|_{L^2}^2 + \frac{\alpha}{16}\|\Lambda^3 u^h\|_{L^2}^2  + \frac{4\beta^2}{\alpha}\|\Lambda^2\sigma^h\|_{L^2}^2,\\
			\nonumber|I_{12}|&\le \frac12\|\Lambda^3 \sigma^h\|_{L^2}^2
			+  C\|\Lambda^2\mathbb{P}(u\cdot \nabla u)^h\|_{L^2}^2 +  \frac{\alpha}{16}\|\Lambda^3 u^h\|_{L^2}^2 + C\|\Lambda\mathbb{P}\udiv(u\cdot \nabla \tau)^h\|_{L^2}^2\\ \label{new7}
			&\le \frac12\|\Lambda^3 \sigma^h\|_{L^2}^2
			+  C\|\nabla^2(u\cdot \nabla u)\|_{L^2}^2 + \frac{\alpha}{16}\|\Lambda^3 u^h\|_{L^2}^2 + C\|\nabla^2(u\cdot \nabla \tau)\|_{L^2}^2.
		\end{align}
		Substituting (\ref{new6}) and (\ref{new7}) into (\ref{highest_order_decay_8}), we can easily deduce that
		\begin{equation}\label{highest_order_decay_9}
			\begin{split}
				\partial_t\langle\Lambda^3 u^h,\Lambda^2 \sigma^h\rangle + \frac{\alpha}{4}\|\Lambda^3 u^h\|_{L^2}^2
				\le (K+\frac12)\|\Lambda^3 \sigma^h\|_{L^2}^2 + \frac{4\beta^2}{\alpha}\|\Lambda^2\sigma^h\|_{L^2}^2 + C(1+t)^{-5},
			\end{split}
		\end{equation} for all $t\ge0$.
		
		We define that
		\begin{equation*}
			\begin{split}
				\mathcal{H}_5(t): = \alpha\|\nabla^3 u^h\|_{L^2}^2 + K\|\nabla^3\tau\|_{L^2}^2 + \eta_3\langle\Lambda^3 u^h,\Lambda^2 \sigma^h\rangle.
			\end{split}
		\end{equation*}
		Then
		\begin{equation}\label{highest_order_decay_10}
			\frac12\alpha\|\nabla^3 u^h\|_{L^2}^2 + \frac12K\|\nabla^3\tau\|_{L^2}^2\le\mathcal{H}_5(t)\le 2\alpha\|\nabla^3 u^h\|_{L^2}^2 + 2K\|\nabla^3\tau\|_{L^2}^2.
		\end{equation}	
		Summing $2\times$\eqref{highest_order_decay_7} and $\eta_3\times$\eqref{highest_order_decay_9} up,
		and using \eqref{highest_order_decay_10} and the smallness of $\eta_3$, we have
		\begin{equation}\label{new8}
			\udt \mathcal{H}_5(t) + \bar{c}_4\mathcal{H}_5(t)\le C\,(1+t)^{-5},
		\end{equation}
		for some positive constants $\bar{c}_4$ and $C$, where we have used
		\begin{equation*}
			\|\Lambda^3 \sigma^h\|_{L^2}^2 +\|\Lambda^2\sigma^h\|_{L^2}^2\le C\|\Lambda^3 \tau\|_{L^2}^2.
		\end{equation*}
		Then, (\ref{new8}) yields (\ref{highest_order_decay_0}). The proof of Lemma \ref{lemma_upper_decay_5} is complete.
		
	\end{proof}
	
	\bigskip
	
	With Lemmas \ref{lemma_upper_decay+1}, \ref{lemma_upper_decay_2}, \ref{lemma_upper_decay_3}, \ref{lemma_upper_decay_4}, and \ref{lemma_upper_decay_5}, and Corollaries \ref{cor1} and \ref{cor2}, we get (\ref{opti1}) and (\ref{opti2}) in Theorem \ref{thm_OB_d_decay}.

	\subsection{Lower time-decay estimates}
	To finish the proof of Theorem \ref{thm_OB_d_decay}, we will establish the lower decay estimates for the system \eqref{Oldroyd_B_d}.
	\begin{lemma}\label{lower_bound}
		Under the assumptions of Theorem \ref{thm_OB_d_decay}, there exists a positive time $t_1$, such that the following estimates:
		\begin{equation}\label{lower_bound_1}
			\|\nabla^k u(t)\|_{L^2} \ge \frac{1}{C} (1 + t)^{-\frac{1}{2}-\frac{k}{2}}, \quad k=0,1,2,3,
		\end{equation}
		and
		\begin{equation}\label{lower_bound_2}
			\|\nabla^k \tau (t)\|_{L^2} \ge \frac{1}{C} (1 + t)^{-1-\frac{k}{2}}, \quad k=0, 1, 2, 3,
		\end{equation}
		hold for all $t\geq t_1$.
	\end{lemma}
	\begin{proof} Recalling from (\ref{Greenfunction_13-1}) and (\ref{Greenfunction_13-2}), we have
		\begin{equation*}
			\begin{split}
				\hat{u}(t) =& \mathcal{G}_3 \hat{u}_0 + K|\xi|\mathcal{G}_1\hat{\sigma}_0 + \int_0^t\mathcal{G}_3(t-s)\hat{\mathcal{F}}_1 (s) + K|\xi|\mathcal{G}_1(t-s)\hat{\mathcal{F}}_2 (s)
				{\rm d}s,\\
				\hat{\sigma}(t) =& -\frac{\alpha}{2}|\xi|\mathcal{G}_1 \hat{u}_0 + \mathcal{G}_2\hat{\sigma}_0 + \int_0^t-\frac{\alpha}{2}|\xi|\mathcal{G}_1(t-s)\hat{\mathcal{F}}_1 (s) + \mathcal{G}_2(t-s)\hat{\mathcal{F}}_2 (s){\rm d}s,
			\end{split}
		\end{equation*}
		where
		\begin{eqnarray*}
			\mathcal{F}_1=-\mathbb{P}\left(u\cdot\nabla u\right),\
			\mathcal{F}_2=-\Lambda^{-1}\mathbb{P}{\rm div}\left(u\cdot\nabla\tau\right).
		\end{eqnarray*}
		From \eqref{lemma_Greenfunction_12} and \eqref{lemma_Greenfunction_13}, we can obtain  for all $t\geq t_1$, that
		\begin{equation}\label{lower_bound_3}
			\begin{split}
				\|\nabla^ku(t)\|_{L^2} = &\left\||\xi|^k\hat{u}(t)\right\|_{L^2}\\
				\ge\,&\frac{1}{C}(1 + t)^{-\frac12-\frac{k}{2}}	-
				C\int_0^t \Big(\int_{|\xi|\leq R'}\Big||\xi|^k\mathcal{G}_3(t-s)\hat{\mathcal{F}}_1 (\xi,s)\\& + |\xi|^{k+1}\mathcal{G}_1(t-s)\hat{\mathcal{F}}_2 (\xi,s)\Big|^2{\rm d}\xi\Big)^\frac12{\rm d}s,
			\end{split}
		\end{equation}
		and	
		\begin{equation}\label{lower_bound_4}
			\begin{split}
				\|\nabla^k\sigma(t)\|_{L^2}= &\left\||\xi|^k\hat{\sigma}(t)\right\|_{L^2}\\
				\ge\, &\frac{1}{C}(1 + t)^{-1-\frac{k}{2}}-
				C\int_0^t \Big(\int_{|\xi|\leq R'}\Big||\xi|^{k+1}\mathcal{G}_1(t-s)\hat{\mathcal{F}}_1 (\xi,s)\\ &+ |\xi|^k\mathcal{G}_2(t-s)\hat{\mathcal{F}}_2 (\xi,s)\Big|^2{\rm d}\xi\Big)^\frac12{\rm d}s.
			\end{split}
		\end{equation}	
	
		Now we estimate the nonlinear term in \eqref{lower_bound_3} for $k=0, 1, 2, 3$. In fact,	using Lemmas \ref{lemma_Greenfunction_4} and \ref{lemma_upper_decay}, we obtain that for $k=0, 1, 2$,
		\begin{equation}\label{lower_bound_5}
			\begin{split}
				&  \int_0^t \Big(\int_{|\xi|\leq R'}\Big||\xi|^k\mathcal{G}_3(t-s)\hat{\mathcal{F}}_1 (\xi,s) + |\xi|^{k+1}\mathcal{G}_1(t-s)\hat{\mathcal{F}}_2 (\xi,s)\Big|^2{\rm d}\xi\Big)^\frac12{\rm d}s\\
				\le C &\int_0^{\frac{t}{2}} \left(\|\frac{1}{|\xi|} \hat{\mathcal{F}}_1 (\xi,s)\|_{L^\infty}   + \|\hat{\mathcal{F}}_2 (\xi,s)\|_{L^\infty}\right)\left(\int_{|\xi|\leq R'} |\xi|^{2(k+1)}e^{-2\theta|\xi|^2(t-s )} {\rm d} \xi\right)^{\frac{1}{2}}{\rm d}s \\
				+ &C\int_{\frac{t}{2}}^t \left(\int_{|\xi|\leq R'} e^{-2\theta|\xi|^2(t-s )}\big(|\xi|^{2k}|\hat{\mathcal{F}}_1(\xi,s)|^2 + |\xi|^{2k}|\hat{\mathcal{F}}_2(\xi,s)|^2\big) {\rm d} \xi\right)^{\frac{1}{2}}{\rm d}s \\
				\le C &\int_0^{\frac{t}{2}}(1+s)^{-1}(1+t-s)^{-1-\frac{k}{2}}{\rm d}s + C\int_{\frac{t}{2}}^t(1+s)^{-\frac{3}{2}-\frac{k}{2}}(1+t-s)^{-\frac{1}{2}}{\rm d}s\\
				\le C& \,(1+t)^{-\frac{3}{4}-\frac{k}{2}}.
			\end{split}
		\end{equation}
		For $k=3$, we have
		\begin{equation}\label{lower_bound_7}
			\begin{split}
				&  \int_0^t \Big(\int_{|\xi|\leq R'}\Big||\xi|^3\mathcal{G}_3(t-s)\hat{\mathcal{F}}_1 (\xi,s) + |\xi|^{4}\mathcal{G}_1(t-s)\hat{\mathcal{F}}_2 (\xi,s)\Big|^2{\rm d}\xi\Big)^\frac12{\rm d}s\\
				\le C &\,(1+t)^{-\frac{9}{4}}
				+ C\int_{\frac{t}{2}}^t \left(\big\||\xi|^{2}\hat{\mathcal{F}}_1(\cdot,s)\big\|_{L^\infty} + \big\||\xi|^{2}\hat{\mathcal{F}}_2(\cdot,s)\big\|_{L^\infty}\right)(1+t-s)^{-1}{\rm d}s \\
				\le C& \,(1+t)^{-\frac{9}{4}} + C\int_{\frac{t}{2}}^t(1+s)^{-\frac{5}{2}}(1+t-s)^{-1}{\rm d}s\\
				\le C& \,(1+t)^{-\frac{9}{4}}.
			\end{split}
		\end{equation}
	
		Next, we turn to estimate the nonlinear term in \eqref{lower_bound_4} for $k=0, 1, 2, 3$. Similar to \eqref{lower_bound_5}, using Lemmas \ref{lemma_Greenfunction_4} and \ref{lemma_upper_decay}, we deduce that for  $k=0, 1, 2$,
		\begin{equation}\label{lower_bound_11}
			\begin{split}
				&  \int_0^t \underbrace{\Big(\int_{|\xi|\leq R'}\Big||\xi|^{k+1}\mathcal{G}_1(t-s)\hat{\mathcal{F}}_1 (\xi,s) + |\xi|^k\mathcal{G}_2(t-s)\hat{\mathcal{F}}_2 (\xi,s)\Big|^2{\rm d}\xi\Big)^\frac12}_{N(k,s)}{\rm d}s
				\\=&\int_0^{\frac{t}{2}} N(k,s){\rm d}s + \int_{\frac{t}{2}}^{t} N(k,s){\rm d}s.
			\end{split}
		\end{equation}For the first term of  \eqref{lower_bound_11}, we have
		\begin{align}\label{lower_bound_6}
			\int_0^{\frac{t}{2}} N(k,s){\rm d}s \le C &\int_0^{\frac{t}{2}} \left(\|\frac{1}{|\xi|} \hat{\mathcal{F}}_1 (\cdot,s)\|_{L^\infty}  +\|\hat{\mathcal{F}}_2 (\cdot,s)\|_{L^\infty}\right)\left(\int_{|\xi|\leq R'} |\xi|^{2k+4}e^{-2\theta|\xi|^2(t-s )}{\rm d} \xi \right)^{\frac{1}{2}}{\rm d}s\notag\\
			+ &C\int_{0}^{\frac{t}{2}}\||\xi|^{k}\hat{\mathcal{F}}_2(\cdot,s)\|_{L^\infty}\left(\int_{|\xi|\leq R'} e^{-\beta(t-s)} {\rm d} \xi\right)^{\frac{1}{2}}{\rm d}s\\
			\le C &\int_0^{\frac{t}{2}}(1+s)^{-1}(1+t-s)^{-\frac{3}{2}-\frac{k}{2}}{\rm d}s  +C\,(1+t)^{-\frac{3}{2}-\frac{k}{2}}\notag\\
			\le C&\, (1+t)^{-\frac{5}{4}-\frac{k}{2}}\notag.
		\end{align}For the second term of  \eqref{lower_bound_11}, we have
		\begin{align}
			\int_{\frac{t}{2}}^{t} N(k,s){\rm d}s\le &\,C\int_{\frac{t}{2}}^t \left(\||\xi|^{k}\hat{\mathcal{F}}_1(\cdot,s)\|_{L^\infty} + \||\xi|^{k}\hat{\mathcal{F}}_2(\cdot,s)\|_{L^\infty}\right)\left(\int_{|\xi|\leq R'} |\xi|^{2}e^{-2\theta|\xi|^2(t-s )} {\rm d} \xi\right)^{\frac{1}{2}}{\rm d}s \notag\\
			&+ C\int_{\frac{t}{2}}^t\||\xi|^{k}\hat{\mathcal{F}}_2(\cdot,s)\|_{L^\infty}\left(\int_{|\xi|\leq R'} e^{-\beta(t-s)} {\rm d} \xi\right)^{\frac{1}{2}}{\rm d}s\\
			\le  &C\int_{\frac{t}{2}}^t(1+s)^{-\frac{3}{2}-\frac{k}{2}}(1+t-s)^{-1}{\rm d}s +C\,(1+t)^{-\frac{3}{2}-\frac{k}{2}}\notag\\
			\le &C\, (1+t)^{-\frac{5}{4}-\frac{k}{2}}\notag.
		\end{align}
		For $k=3$, we have
		\begin{equation}\label{lower_bound_8}
			\begin{split}
				&  \int_0^t \Big(\int_{|\xi|\leq R'}\Big||\xi|^{4}\mathcal{G}_1(t-s)\hat{\mathcal{F}}_1 (\xi,s) + |\xi|^3\mathcal{G}_2(t-s)\hat{\mathcal{F}}_2 (\xi,s)\Big|^2{\rm d}\xi\Big)^\frac12{\rm d}s\\
				\le C&\,(1+t)^{-\frac{11}{4}}
				+ C\int_{\frac{t}{2}}^t \left(\||\xi|^{2}\hat{\mathcal{F}}_1(\cdot,s)\|_{L^2} + \||\xi|^{2}\hat{\mathcal{F}}_2(\cdot,s)\|_{L^2}\right)(1+t-s)^{-1}{\rm d}s \\
				\le C& \,(1+t)^{-\frac{11}{4}} + C\int_{\frac{t}{2}}^t\Big(\|\nabla^{2}(u\cdot\nabla u)(s)\|_{L^2} + \|\nabla^{2}(u\cdot\nabla\tau)(s)\|_{L^2}\Big)(1+t-s)^{-1}{\rm d}s\\
				\le C& \,(1+t)^{-\frac{11}{4}},
			\end{split}
		\end{equation}
		where we have used the estimate:
		\begin{equation*}\label{lower_bound_9}
			\begin{split}
				&\|\nabla^{2}(u\cdot\nabla u)(t)\|_{L^2} + \|\nabla^{2}(u\cdot\nabla\tau)(t)\|_{L^2}\\
				\le &C\big(\| u\|_{L^2}^{\frac12}\|\nabla^2 u\|_{L^2}^{\frac12}\|\nabla^3 u\|_{L^2} + \|\nabla u\|_{L^2}^{\frac12}\|\nabla^3 u\|_{L^2}^{\frac12}\|\nabla^2 u\|_{L^2}\big)\\
				&+C\big(\| u\|_{L^2}^{\frac12}\|\nabla^2 u\|_{L^2}^{\frac12}\|\nabla^3 \tau\|_{L^2} + \|\nabla u\|_{L^2}^{\frac12}\|\nabla^3 u\|_{L^2}^{\frac12}\|\nabla^2 \tau\|_{L^2}+  \|\nabla \tau\|_{L^2}^{\frac12}\|\nabla^3 \tau\|_{L^2}^{\frac12}\|\nabla^2 u\|_{L^2}\big)\\
				\le &C (1+t)^{-3},\\
			\end{split}
		\end{equation*}for all $t\geq 0$, which is similar to \eqref{new_H5_L2_9}.
		
		Using \eqref{lower_bound_3} -- \eqref{lower_bound_8}, and the fact that $\|\nabla^k\sigma\|_{L^2}\le C \|\nabla^k\tau\|_{L^2},$ we conclude that \eqref{lower_bound_1} and \eqref{lower_bound_2} hold for $t\ge t_1$ ($t_1$ is sufficiently large), i.e., \eqref{opti3} and \eqref{opti4} hold. Thus the proof of Theorem \ref{thm_OB_d_decay} is complete.
	\end{proof}
	\section*{Acknowledgments}This work was supported by the Guangdong Basic and Applied Basic Research Foundation $\#2020B1515310015$ and \#2022A1515012112, and by the National Natural Science Foundation of China $\#12071152,$ and by the Guangdong Provincial Key Laboratory of Human Digital Twin (\#2022B1212010004).

\end{document}